\RequirePackage{silence}
\documentclass[12pt,reqno]{amsart}
\usepackage{amsfonts}
\usepackage{amssymb, amsmath, latexsym}
\usepackage[T1]{fontenc}
\usepackage{fouriernc}
\usepackage{longtable}
\usepackage{booktabs} 
\usepackage{tabularx} 
\usepackage{lipsum}
\usepackage{amsthm}
\usepackage{array}
\usepackage{xcolor}
\usepackage[colorlinks=true]{hyperref}
\hypersetup{
    linkcolor=blue!70!black,   
    citecolor=red!70!black,    
    urlcolor=blue!70!black,    
    linktoc=page,              
    bookmarksnumbered=true,    
    pdfpagemode=UseOutlines    
}
\usepackage{mleftright}
\usepackage{tikz-cd}
\usepackage[numbers,sort&compress]{natbib}
\usepackage{graphicx,mathtools}
\usepackage[
    top=1.5cm,    
    bottom=1.5cm, 
    left=1.5cm,     
    right=1.5cm,    
    includehead,  
    includefoot
]{geometry}
\usepackage{mathtools}
\usepackage{array}
\usepackage{graphicx,mathtools}
\usepackage{ragged2e}
\usepackage{stmaryrd}
\usepackage{etoolbox}
\usepackage{microtype}
\setcitestyle{square}
\SetSymbolFont{stmry}{bold}{U}{stmry}{m}{n}
\allowdisplaybreaks
\numberwithin{equation}{section}
\theoremstyle{plain}
\newtheorem{thm}{Theorem}[section]
\newtheorem{lem}[thm]{Lemma}
\newtheorem{prop}[thm]{Proposition}
\newtheorem{cor}[thm]{Corollary}

\theoremstyle{definition}
\newtheorem{df}[thm]{Definition}
\newtheorem{ex}[thm]{Example}

\newtheorem{conj}[thm]{Conjecture}
\newtheorem{question}[thm]{Question}

\theoremstyle{remark}
\newtheorem{rmk}[thm]{Remark}

\newcommand{\ZZ}{\mathbb{Z}}
\newcommand{\NN}{\mathbb{N}}
\newcommand{\QQ}{\mathbb{Q}}
\newcommand{\RR}{\mathbb{R}}
\newcommand{\CC}{\mathbb{C}}

\newcommand{\FF}{\mathbb{F}}
\newcommand{\GG}{\mathbb{G}}
\newcommand{\KK}{\mathbb{K}}
\newcommand{\PP}{\mathbb{P}}

\newcommand{\nfk}{\mathfrak{n}}

\newcommand{\Dfk}{\mathfrak{D}}

\newcommand{\Acal}{\mathcal{A}}

\newcommand{\Dcal}{\mathcal{D}}

\newcommand{\Mcal}{\mathcal{M}}

\newcommand{\Ocal}{\mathcal{O}}

\newcommand{\Rcal}{\mathcal{R}}

\newcommand{\id}{\operatorname{id}}

\renewcommand{\ker}{\operatorname{ker}}

\newcommand{\dep}{\operatorname{dep}}
\newcommand{\wt}{\operatorname{wt}}

\let\oldforall\forall
\renewcommand{\forall}{\oldforall \: }
\let\oldexist\exists
\renewcommand{\exists}{\oldexist \: }

\let\subset\subseteq

\DeclarePairedDelimiter\floor{\lfloor}{\rfloor}

\newcommand{\sbb}[1]{(\kern-.2em({#1})\kern-.2em)}
\newcommand{\mbb}[1]{[\![{#1}]\!]}
\newcommand{\cpi}{\tilde{\pi}}

\newcommand{\ww}[1]{\mathfrak{#1}} 
\newcommand{\C}{\mathbf{C}}
\renewcommand{\H}{\mathbf{H}}
\newcommand{\BC}{\mathrm{BC}}
\newcommand{\dBC}{\mathrm{dBC}}

\newcommand{\bfC}{\mathbf{C}}

\newcommand{\Ical}{\mathbf{I}}
\newcommand{\Zcal}{\mathfrak{Z}}
\newcommand{\ad}{\mathrm{ad}}
\newcommand{\ext}{\mathrm{ext}}
\title[On {\protect\lowercase{$u$}}-Multiple Zeta Values in Positive Characteristic]
{On {\protect\lowercase{$u$}}-Multiple Zeta Values in Positive Characteristic}
\author[H.-C. Tsui]{Hung-Chun Tsui}
\address{Department of Mathematics, National Tsing Hua University, Hsinchu 300044, Taiwan}
\email{hctsui@gapp.nthu.edu.tw}
\date{\today}
\dedicatory{}
\commby{}
\begin{document}
\begin{abstract}
In this paper, we introduce the concepts of the $u$-bracket, finite multiple harmonic $u$-series, and $u$-multiple zeta values via the Carlitz module. These objects serve as function field counterparts to the classical theory of $q$-analogs. We prove that the ``limits'' of finite multiple harmonic $u$-series at Carlitz torsion points yield Thakur's multiple zeta values and finite multiple zeta values over $\mathbb{F}_r(\theta)$ from analytic and algebraic perspectives, respectively. This can be regarded as a positive characteristic analog of the results by Bachmann, Takeyama, and Tasaka \cite{BTT18}. Furthermore, we investigate the properties of $u$-multiple zeta values and their expansions, obtaining a family of explicit relations among Thakur's multiple zeta values at both positive and non-positive indices.
\end{abstract}
\maketitle
\tableofcontents
\section{Introduction}
\subsection{Classical \texorpdfstring{$q$}{q}-Analog Theory}

Let $\NN$ denote the set of positive integers and $\ZZ$ denote the set of integers. We begin by introducing the classical $q$-analog theory. In various mathematical frameworks, $q$-analogs serve as deformations of classical formulas or theorems, such that the original theory is retrieved in the limit as $q \to 1$. The concept of $q$-analogs has important applications in various mathematical fields, especially in number theory, partition theory, and special functions (see \cite{GR2004, andrews1976theory} for more details).

As basic examples, we define the $q$-analogs of integers and factorials. For any non-negative integer $n$, the \textit{$q$-bracket} is defined as
$$[n]_q := \frac{1-q^n}{1-q},$$
which satisfies $\lim_{q \to 1} [n]_q = n$. Based on this, the \textit{$q$-factorial} is defined as
$$[n]_q! = [1]_q [2]_q \cdots [n]_q$$
for $n \geq 1$, with $[0]_q! = 1$. The $q$-factorial has many properties similar to the classical factorial. For example, the $q$-factorial admits a factorization into cyclotomic polynomials $\Phi_k(q)$ given by$$[n]_q! = \prod_{\substack{k\in\NN\\k\neq 1}} \Phi_k(q)^{\lfloor n/k \rfloor}.$$
As $q \to 1$, this identity recovers Legendre's formula
$$
n!=\prod_{p\text{ prime}}p^{\sum_{e\geq 1}\floor{n/p^e}},
$$
since $[n]_q! \to n!$ and the value $\Phi_k(1)$ is $p$ if $k=p^e$ and $1$ otherwise.
\subsection{Classical Theory of Multiple Zeta Values} We now turn our attention to the theory of multiple zeta values. We denote by $$\Ical=\bigcup_{m=1}^\infty \NN^m\cup\{\varnothing\}$$  the set of indices and by $$\Ical^{\ad}=\{(s_1,\ldots,s_m)\in\Ical:s_1\geq 2\}\cup\{\varnothing\}$$ the set of \textit{admissible} indices. For a non-empty index $\ww{s}=(s_1,\ldots,s_m)\in\Ical$, we define its \textit{depth} and \textit{weight} by
\[
\dep(\ww{s})=m\quad \text{and} \quad \wt(\ww{s})=s_1+\cdots+s_m,
\]
respectively.
For the empty index $\varnothing$, we put $\dep(\varnothing)=0=\wt(\varnothing)$. Furthermore, throughout this paper, an empty sum is defined to be zero and an empty product to be one.

We first introduce the theory of \textit{finite multiple harmonic $q$-series}, following the framework in \cite{BTT18}. The classical multiple zeta value is defined by
$$\zeta(\ww{s}) = \sum_{n_1 > \cdots > n_m \geq 1} \frac{1}{n_1^{s_1} \cdots n_m^{s_m}} \in \mathbb{R}, \quad \varnothing\neq \ww{s}=(s_1,\ldots,s_m) \in \Ical^{\text{ad}}.$$
We put $\zeta(\varnothing)=1$ and denote the $\mathbb{Q}$-algebra generated by all multiple zeta values by $\Zcal = \langle \zeta(\ww{s}) : \ww{s} \in \Ical^{\text{ad}} \rangle_{\mathbb{Q}}$. 
Besides this, we also consider two other variants called the \textit{finite multiple zeta values} and \textit{symmetric multiple zeta values} in the following.

Consider the ring 
$$\mathcal{A} = \left(\prod_p \mathbb{F}_p\right) / \left(\bigoplus_p \mathbb{F}_p\right),$$
where the product and direct sum are taken over all primes $p$. This ring was introduced by Kontsevich \cite{Kontsevich2009}, and it admits a natural embedding $\mathbb{Q} \hookrightarrow \mathcal{A}$ via the diagonal map. The finite multiple zeta value, first introduced by Kaneko and Zagier \cite{KanekoZagier}, is defined by
$$\zeta_{\Acal}(\ww{s})=\left(\sum_{p>n_1>\cdots>n_m\geq 1}\frac{1}{n_1^{s_1}\cdots n_m^{s_m}}\mod{p}\right)_p\in\Acal,\quad \varnothing\neq \ww{s}=(s_1,\ldots,s_m)\in\Ical.$$
We put $\zeta_{\Acal}(\varnothing)=(1)_p$ and denote the $\QQ$-algebra generated by all finite multiple zeta values by $\Zcal_\Acal=\langle \zeta_{\Acal}(\ww{s}):\ww{s}\in\Ical\rangle_\QQ$. 

On the other hand, following the framework of Ihara, Kaneko and Zagier \cite[Proposition 1]{IKZ06}, we have the notion of \textit{stuffle-regularized multiple zeta values} $$Z^\ast(\ww{s};T) \in \mathbb{R}[T],\quad \ww{s}\in\Ical$$ satisfying
\[
Z^\ast(\ww{s};T) = \zeta(\ww{s})\quad \text{for } \ww{s}\in\Ical^{\mathrm{ad}}, \quad\text{and}\quad Z^\ast(1;T) = T
\]
where $T$ denotes a formal variable.
Then the symmetric multiple zeta value, also introduced by Kaneko and Zagier \cite[(87), (88)]{KanekoZagier}, is defined by 
$$\zeta_{S}(\ww{s})=\sum_{j=0}^m(-1)^{s_1+\cdots+s_j}Z^\ast(s_j,\ldots,s_1;T)Z^\ast(s_{j+1},\ldots,s_m;T)\pmod {\zeta(2)}$$
where $\varnothing\neq \ww{s}=(s_1,\ldots,s_m)\in\Ical$. We put $\zeta_S(\varnothing)=1\pmod{\zeta(2)}$. Here, we mention that by \cite[Theorem 3]{KanekoZagier} (see also \cite{Kaneko2019}), the symmetric multiple zeta value does not depend on
$T$ and lies in $\Zcal$.

We recall the famous Kaneko--Zagier conjecture.
\begin{conj}[Kaneko--Zagier Conjecture]
    There is a $\QQ$-algebra isomorphism $$\varphi_{\textrm{KZ}}:\Zcal_{\Acal}\to\Zcal/\zeta(2)\Zcal$$ such that $\varphi_{\textrm{KZ}}(\zeta_\Acal(\ww{s}))=\zeta_S(\ww{s})$
    for all $\ww{s}\in\Ical$.
\end{conj}

In \cite{BTT18}, Bachmann, Takeyama, and Tasaka provide evidence for the Kaneko--Zagier conjecture by considering the finite multiple harmonic $q$-series
\[
Z_{<n}(\ww{s};q) = \sum_{n > n_1 > \cdots > n_m > 0} 
\frac{q^{(s_1-1)n_1 + \cdots + (s_m-1)n_m}}{[n_1]_q^{s_1} \cdots [n_m]_q^{s_m}},
\]
where $\varnothing \neq \ww{s}=(s_1,\ldots,s_m) \in \Ical$. Here, $q \in \CC$ is assumed to satisfy $q^{k} \neq 1$ for all $0 < k < n$. We also set $Z_{<n}(\ww{s}; q) = 0$ if $\dep(\ww{s}) \geq n$ and $Z_{<n}(\varnothing; q) = 1$. This model of $q$-series is often called the Bradley--Zhao model (see \cite{Bradley2005, Zhao2007}). The authors in \cite{BTT18} showed that the ``limits'' of $Z_{<n}(\ww{s};q)$ evaluated at the $n$-th root of unity $\zeta_n:=\exp(2\pi\sqrt{-1}/n)$ relate to both finite and symmetric multiple zeta values. Precisely, they proved the following:
\begin{thm}[{\cite[Theorem 1.1]{BTT18}}]\label{thm:1.1}
    For any index $\ww{s}\in\Ical$, we have
    $$\zeta_\Acal(\ww{s})=\left(Z_{<p}(\ww{s};\zeta_p)\mod (1-\zeta_p)\right)_p\in\Acal.$$
\end{thm}
\begin{thm}[{\cite[Theorem 1.2]{BTT18}}]\label{thm:1.2}
    For any index $\ww{s}\in\Ical$, we have
    $$\zeta_S(\ww{s})\equiv\Re\left(\lim_{n\to\infty}Z_{<n}(\ww{s};\zeta_n)\right)\pmod {\zeta(2)\Zcal}$$
    where $\Re(\cdot)$ denotes the real part of a complex number.
\end{thm}

Since $\zeta_\Acal(\ww{s})$ and $\zeta_S(\ww{s})$ arise from "limits" of the same object, it is expected that they satisfy the same algebraic relations, providing evidence for the Kaneko--Zagier conjecture. Furthermore, Bachmann, Takeyama and Tasaka used the star versions of Theorems \ref{thm:1.1} and \ref{thm:1.2} to re-establish duality formulas for both \textit{finite multiple zeta star values} and \textit{symmetric multiple zeta star values} (see \cite[Theorem 1.3]{BTT18}). We also mention that similar phenomena for variants of finite and symmetric multiple zeta values have been investigated in \cite{BTT21} and \cite{Tasaka21}.

On the other hand, for any index $\ww{s}\in\Ical^{\ad}$, the limit of $Z_{<n}(\ww{s};q)$ as $n\to\infty$, denoted by $\zeta_q(\ww{s})$, is called the \textit{$q$-multiple zeta value}. Note that this limit is well-defined for $|q|<1$, and as $q\to 1^-$, we have $\zeta_q(\ww{s})\to \zeta(\ww{s})$ for any index $\ww{s}\in\Ical^{\ad}$. This model of $q$-multiple zeta values inherits many relations from classical multiple zeta values (see \cite{Bradley2005, OkudaTakeyama2007, Zhao2007}). In particular, it admits a modified stuffle relation arising from the identity
$$\frac{q^{(s-1)n}}{[n]_q^s}\frac{q^{(r-1)n}}{[n]_q^r}=(1-q)\frac{q^{(s+r-2)n}}{[n]_q^{s+r-1}}+\frac{q^{(s+r-1)n}}{[n]_q^{s+r}},\quad n,s,r\in\NN$$
(see \cite[(2.2)]{Bradley2005}). We also remark that there are various other models of $q$-analogs of multiple zeta values in the literature (see, e.g., \cite{Schlesinger2001, Zudilin2003, Okounkov2014, Bachmann2020}).
\subsection{Preliminaries on Function Field Arithmetic}
We now consider the positive characteristic setting: \vspace{0.2cm}\\
\noindent
\begin{tabularx}{\textwidth}{@{} l X @{}}
    \toprule
    \textbf{Symbol} & \textbf{Description} \\
    \midrule
    $\mathbb{F}_r$ & finite field with $r$ elements, where $r$ is a power of a prime $p$ \\
    $A = \mathbb{F}_r[\theta]$ & the ring of polynomials in the variable $\theta$ over $\mathbb{F}_r$ \\
    $K = \mathbb{F}_r(\theta)$ & the field of fractions  of $A$\\
    $K_\infty=\mathbb{F}_r((1/\theta))$ & the completion of $K$ at the infinite place with uniformizer $1/\theta$ \\
    $\mathbb{C}_\infty$ & the completion of a fixed algebraic closure of $K_\infty$ \\
    $| \cdot |_{\infty}$ & the $\infty$-adic absolute value on $\CC_\infty$ normalized such that $|\theta|_\infty=r$.  \\
    $A_+$ & the set of monic polynomials in $A$\\
    $A_{d}$& the set of polynomials in $A$ with $\deg=d$\\
    $A_{<d}$& the set of polynomials in $A$ with $\deg<d$\\
    $A_{\leq d}$& the set of polynomials in $A$ with $\deg\leq d$\\
    $A_{+,d}=A_+\cap A_d$ &the set of monic polynomials in $A$ with $\deg=d$\\
    $A_{+,<d}=A_+\cap A_{<d}$ &the set of monic polynomials in $A$ with $\deg<d$\\
    $A_{+,\leq d}=A_+\cap A_{\leq d}$ &the set of monic polynomials in $A$ with $\deg\leq d$\\
    \bottomrule
\end{tabularx}
\vspace{0.2cm}

Note that in the function field setting, the objects $A$, $K$, $K_\infty$, and $\mathbb{C}_\infty$ play the roles of $\mathbb{Z}$, $\mathbb{Q}$, $\mathbb{R}$, and $\mathbb{C}$ respectively. Furthermore, $A_+$ is viewed as the function field analog of the set of positive integers $\NN$.

We first recall the well-known Carlitz theory (see \cite{Carlitz1935,Goss1996,Thakur2004} for details). For any commutative ring $R$ with unity, let $R$\textbf{-alg} and $R$\textbf{-mod} denote the category of commutative associative $R$-algebras and the category of $R$-modules, respectively. Then the \textit{Carlitz module} is defined as the functor
\[ \C: A\textbf{-alg} \to A\textbf{-mod} \]
which sends each $A$-algebra $B$ to the $A$-module $\C(B)=B$ where the $A$-module structure is given by the unique $\FF_r$-linear homomorphism 
\[ \C: A \to \operatorname{End}_{\mathbb{F}_r}(\GG_a(B)), \quad a\mapsto \C_a(\tau):=\sum_{i=0}^{\deg a}[a,i]\tau^i, \]
such that $\C_\theta(\tau)=\theta\id+\tau$. Here, $\operatorname{End}_{\FF_r}(\GG_a(B))$ denotes the ring of all $\FF_r$-linear endomorphisms of the additive group $\GG_a$ over $B$, and $\tau:u\mapsto u^r$ denotes the $r$-th power Frobenius endomorphism on $B$. Furthermore, for any $a\in A$, we let 
\[ \C_a(X)=\sum_{i=0}^{\deg a}[a,i]X^{r^i}\in A[X] \]
be the associated $\FF_r$-linear polynomial.

The Carlitz module $\C$ serves as the function field counterpart of the multiplicative group 
\[ \GG_m:\ZZ\textbf{-alg} \to \ZZ\textbf{-mod}. \]
They share many similar properties. For example, we have the \textit{Carlitz exponential} defined by
\[ \exp_\C(X)=\sum_{i=0}^\infty \frac{X^{r^i}}{D_i},\quad D_i=\prod_{j=0}^{i-1}(\theta^{r^i}-\theta^{r^j}), \]
which satisfies the functional equation  
$$\C_a(\exp_\C(X))=\exp_\C(aX),\quad a\in A,$$
and plays the role of the classical exponential $\exp(z)$. This induces the following short exact sequence of $A$-modules:
\[
0 \longrightarrow \tilde{\pi} A \longrightarrow \operatorname{Lie}(\C)(\CC_\infty)=\CC_\infty \xrightarrow{\exp_{\C}} \C(\CC_\infty)=\CC_\infty \longrightarrow 0,
\]
which is analogous to the short exact sequence of $\ZZ$-modules:
\[
0 \longrightarrow 2\pi \sqrt{-1} \mathbb{Z} \longrightarrow \operatorname{Lie}(\GG_m)(\CC)=\mathbb{C} \xrightarrow{\exp} \GG_m(\CC)=\mathbb{C}^\times \longrightarrow 1.
\]
Here, $\cpi$ is the \textit{Carlitz period} given by
\[ {\cpi}=(-\theta)^{r/(r-1)}\prod_{i=1}^\infty \left(1-\theta^{1-r^i}\right)^{-1}\in\CC_\infty, \]
where a choice of an $(r-1)$-st root of $-\theta$ is fixed (thus, $\cpi$ is well-defined up to a factor in $\FF_r^\times$). We remark that $\cpi$ is transcendental (see \cite{Wade1941}) and naturally serves as the function field counterpart of $2\pi \sqrt{-1}$. 

On the other hand, for $\nfk\in A$, the $A$-module of \textit{Carlitz $\nfk$-torsion points}, defined by
\[ \Lambda_\nfk:=\{u\in\CC_\infty:\C_\nfk(u)=0\}, \]
is a cyclic $A$-module generated by the element $\lambda_\nfk:=\exp_\C(\cpi/\nfk)$. This torsion submodule $\Lambda_\nfk$ of $\C(\CC_\infty)$ is completely analogous to the classical group of $n$-th roots of unity in $\GG_m(\CC)=\mathbb{C}^\times$.

Next, we introduce the theory of multiple zeta values in this setting. We denote by 
$$\Ical^{\ext}:=\bigcup_{m=1}^\infty \ZZ^m\cup\{\varnothing\}$$
the set of extended indices. An index $\ww{s}\in\Ical^{\ext}$ is called \textit{positive} if $\ww{s}\in\Ical$. Otherwise, it is called \textit{non-positive}. Then \textit{Thakur's power sum} is defined by
$$S_d(s)=\sum_{a\in A_{+,d}}\frac{1}{a^{s}},\quad d,s\in\ZZ.$$
 By the empty sum convention, we note that $S_{<d}(s)=0$ for $d<0$. Now, for any non-empty index $\ww{s}=(s_1,\ldots,s_m)\in\Ical^{\ext}$, we define \textit{Thakur's multiple power sum}
$$S_d(\ww{s})=\sum_{d=d_1>\cdots>d_m\geq 0}S_{d_1}(s_1)\cdots S_{d_m}(s_m),\quad d\in\ZZ,$$
and
$$S_{<d}(\ww{s})=\sum_{d>d_1>\cdots>d_m\geq 0}S_{d_1}(s_1)\cdots S_{d_m}(s_m),\quad d\in\ZZ.$$
We conventionally set $S_d(\varnothing)=\delta_{d,0}$ and $S_{<d}(\varnothing) = 1$ for all $d\in\ZZ$ where $\delta_{i,j}$ denotes the Kronecker delta. 
With these notions, \textit{Thakur's multiple zeta value} is defined by
$$\zeta_A(\ww{s})=\lim_{d\to\infty}S_{<d}(\ww{s})\in K_\infty$$
for any index $\ww{s}\in\Ical^{\ext}$. We mention that for any non-empty positive $\ww{s}\in\Ical$, we have
$$\zeta_A(\ww{s})=\sum_{\substack{a_1,\ldots,a_m\in A_+\\\deg a_1>\cdots>\deg a_m\geq 0}}\frac{1}{a_1^{s_1}\cdots a_m^{s_m}}\in K_\infty.$$
On the other hand, for any non-positive $\ww{s}\in\Ical^{\ext}$, note that the power sum $S_d(s)$ at non-positive integer $s\leq 0$ vanishes for sufficiently large $d$ (see \cite[Theorem 5.1.2]{Thakur2004}). Thus, $S_d(s)\to 0$ as $d\to\infty$ for any $s\in\ZZ$, which guarantees the convergence of $\zeta_A(\ww{s})$. We denote by $\Zcal_\infty=\langle \zeta_A(\ww{s}):\ww{s}\in\Ical\rangle_K=\langle \zeta_A(\ww{s}):\ww{s}\in\Ical^{\ext}\rangle_K$ the $K$-algebra generated by all Thakur's multiple zeta values (see Lemma \ref{lem:linearcombination-non-positive-index}).

In 2010, Thakur \cite{Thakur2010} proved that the product of two of Thakur's multiple zeta values at positive indices can be expressed as an $\FF_p$-linear combination of values of the same weight, known as the \textit{$r$-shuffle relations}. Chen \cite{Chen2015} proved the following explicit formula:
For any $r_1,s_1 \in \NN$,
\[
\zeta_A(r_1)\zeta_A(s_1)=\zeta_A(r_1,s_1)+\zeta_A(s_1,r_1)+\zeta_A(r_1+s_1)+\sum_{\substack{i,j\geq 1\\i+j=r_1+s_1}}\Delta^{i,j}_{r_1,s_1}\zeta_A(i,j),
\]
where 
\begin{equation}\label{eq:df_Delta}
\Delta_{r_1,s_1}^{i,j}=\begin{cases} (-1)^{r_1-1}\binom{j-1}{r_1-1}+(-1)^{s_1-1}\binom{j-1}{s_1-1}, & \text{if } (r-1)\mid j,\\ 0, & \text{otherwise},
    \end{cases}
\end{equation} 
and $\binom{m}{n}$ denotes the usual binomial coefficient modulo $p$. Based on Chen's formula, Yamamoto observed a (conjectural) inductive formula for the $r$-shuffle relations \eqref{eq:inductive-r-shuffle}. The following $r$-shuffle algebra $(\Rcal,\ast)$ was formulated in \cite{Shi2018}:

\begin{df}\label{df-zeta-values-as-words}
    Let $\Mcal$ be the free monoid generated by the set $\{x_k : k \in \NN\}$ and let $\Rcal$ be the $\FF_p$-vector space generated by $\Mcal$. For any non-empty index $\ww{s}=(s_1,\ldots,s_m)\in\Ical$ and the empty index $\varnothing$, we define
    $\ww{s}^-=(s_2,\ldots,s_m)$ and $\varnothing^-=\varnothing$. Moreover, we denote the corresponding words by
    $x_{\ww{s}}=x_{s_1}\cdots x_{s_m}$ and $x_\varnothing=1.$
    We define the \textit{$r$-shuffle product} $\ast$ on $\Rcal$ inductively on the sum of depths as follows:
    \begin{enumerate}
        \item For any index $\ww{s}\in\Ical$, define
        \[
        1\ast x_{\ww{s}}=x_{\ww{s}}\ast 1=x_\ww{s}.
        \]
        \item For non-empty indices $\ww{r}=(r_1,\ldots,r_m),\ww{s}=(s_1,\ldots,s_n)\in\Ical$, define
        \begin{align}\label{eq:inductive-r-shuffle}
             x_{\ww{r}} \ast x_{\ww{s}}
            = x_{r_1} ( x_{\ww{r}^-} \ast x_{\ww{s}} ) 
            + x_{s_1} ( x_{\ww{r}} \ast x_{\ww{s}^-} ) 
            + x_{r_1 + s_1} ( x_{\ww{r}^-} \ast x_{\ww{s}^-} )+ \sum_{\substack{i,j\geq 1\\i+j=r_1+s_1}} 
            \Delta^{i,j}_{r_1,s_1} x_i ( ( x_{\ww{r}^-} \ast x_{\ww{s}^-} ) \ast x_j ).
        \end{align}
        \item Extend the product $\ast$ to the $\FF_p$-vector space $\Rcal$ by the distributive law.
    \end{enumerate}
\end{df}

Shi showed in her PhD thesis \cite{Shi2018} that the $r$-shuffle product $\ast$ on $\Rcal$ indeed encodes the $r$-shuffle relations of Thakur's multiple zeta values. Precisely, we have the following theorem:

\begin{thm}[{\cite[Theorem 3.1.4]{Shi2018}}]\label{thm:Shi-r-shuffle}
    For $d\geq 0$, let $\widehat{S_{<d}}, \widehat{\zeta_A}:\Rcal \to K_\infty$ be the unique $\FF_p$-linear maps satisfying
    \[
    \widehat{S_{<d}}(1) := 1,
    \quad
    \widehat{S_{<d}}(x_{\ww{s}})
    := S_{<d}(\ww{s}),
    \]
    and
    \[
    \widehat{\zeta_A}(1):=1,
    \quad
    \widehat{\zeta_A} (x_{\ww{s}})
    := \lim_{d\to \infty} \widehat{S_{<d}}(x_{\ww{s}})
    = \zeta_A(\ww{s}).
    \]
    Then $\widehat{S_{<d}}$ is an $\FF_p$-algebra homomorphism.
    Consequently, $\widehat{\zeta_A}$ is also an $\FF_p$-algebra homomorphism, i.e., 
    \[
    \widehat{\zeta_A}(x_{\ww{r}} \ast x_{\ww{s}})=\zeta_A(\ww{r})\zeta_A(\ww{s}).
    \]
\end{thm}

Next, we consider the \textit{finite multiple zeta values over $K$}. We begin with the ring
\[
\Acal_K=\left(\prod_{v}\FF_v\right)\Bigg/\left(\bigoplus_v \FF_v\right),
\]
where $v$ ranges over all monic irreducible polynomials in $A$ and $\FF_v=A/(v)$. Notice that we have a natural embedding $K \hookrightarrow \Acal_K$. The $K$-algebra $\Acal_K$ is a natural analog of Kontsevich's ring $\Acal$ and was studied in \cite{CM17,Shi2018}. The finite multiple zeta value over $K$ is defined by
\[
\zeta_{\Acal_K}(\ww{s})=\left(S_{<\deg v}(\ww{s})\mod{v}\right)_v\in\Acal_K,\quad \varnothing\neq \ww{s}=(s_1,\ldots,s_m)\in\Ical^{\ext}.
\]
We set $\zeta_{\Acal_K}(\varnothing)=(1)_v$ and denote the $K$-algebra generated by all finite multiple zeta values over $K$ by $\Zcal_{\Acal_K}=\langle \zeta_{\Acal_K}(\ww{s}):\ww{s}\in\Ical\rangle_K=\langle \zeta_{\Acal_K}(\ww{s}):\ww{s}\in\Ical^{\ext}\rangle_K$ (see Lemma \ref{lem:linearcombination-non-positive-index}). It is immediate from the definitions and Theorem \ref{thm:Shi-r-shuffle} that the finite multiple zeta values over $K$ at positive indices also satisfy the same $r$-shuffle relations as Thakur's multiple zeta values. 
\subsection{Main Results and Organization of the Paper}
In this paper, we introduce the concepts of the \textit{$u$-bracket} $[a]_u=\C_a(u)/u$, the \textit{finite multiple harmonic $u$-series} $H_{<d}(\ww{s};u)$, and \textit{$u$-multiple zeta values} $\zeta_u(\ww{s})$ (see Definitions \ref{df-u-bracket}, \ref{df:fmhus}, and \ref{df:uMZV}). These objects serve as function field counterparts to the classical $q$-bracket $[n]_q$, the finite multiple harmonic $q$-series $Z_{<n}(\ww{s};q)$, and $q$-multiple zeta values $\zeta_q(\ww{s})$, respectively. Based on these new definitions, we establish the following results:
\begin{thm}[restated as Theorem \ref{thm:limit of u}]\label{thm:main1}
    For any index $\ww{s}\in\Ical^{\ext}$, we have $$\zeta_A(\ww{s})=\lim_{\deg\nfk\to\infty}H_{<\deg\nfk}(\ww{s};\lambda_\nfk).$$
\end{thm}

\begin{thm}[restated as Theorem \ref{thm:u to finite MZV}]\label{thm:main2}
    For any index $\ww{s}\in\Ical^{\ext}$, we have
    $$
        \zeta_{\mathcal{A}_K}(\ww{s})=\left(H_{<\deg v}(\ww{s};\lambda_v)\mod{\lambda_v}\right)_v\in\mathcal{A}_K.
    $$
\end{thm}
Roughly speaking, taking the ``analytic limit'' of the finite multiple harmonic $u$-series at Carlitz torsion points leads to Thakur's multiple zeta values, while taking the ``algebraic limit'' provides the finite multiple zeta values over $K$ (cf. Theorems \ref{thm:1.1} and \ref{thm:1.2}). 

Furthermore, we compute the formal power series expansion of $u$-multiple zeta values and show that their coefficients are in fact $K$-linear combinations of Thakur's multiple zeta values:

\begin{thm}[restated as Propositions \ref{prop:expansion of umzv} and \ref{prop:gamma mzv}]\label{thm:main3}
    Let $\ww{s}=(s_1,\ldots,s_m)\in\Ical^{\ext}$. Then
    $$\zeta_u(\ww{s}) = \sum_{N=0}^{\infty} \gamma_N(\ww{s}) u^{N(r-1)}\in\CC_\infty\mbb{u},$$
    where 
    $$\gamma_N(\ww{s}) = \sum_{\substack{n_1, \dots, n_m \ge 0 \\ n_1 + \dots + n_m = N}} \sum_{i_1=0}^{n_1(r-1)} \cdots \sum_{i_m=0}^{n_m(r-1)} \left( \prod_{j=1}^m c_{n_j,i_j}^{(s_j)} \right) \zeta_A(s_1-i_1, \dots, s_m-i_m),$$
    for some explicit constants $c_{n,i}^{(s)}\in K$.
\end{thm}

We remark that the above theorems hold not only for positive indices but also for non-positive ones.

In addition, following the approach of \cite{Shi2018} (see also \cite{Thakur2010}), we show that the finite multiple harmonic $u$-series $H_{<d}(\ww{s};u)$ and the $u$-multiple zeta values $\zeta_u(\ww{s})$ satisfy the same $r$-shuffle relations as Thakur's multiple zeta values (see Corollaries \ref{cor:Hu shuffle}, \ref{cor:zeta u shuffle}, \ref{cor:Hu shuffle formal} and \ref{cor:zeta u shuffle formal}). As an application, we find a Hasse-Schmidt derivation $(\Dcal_N)_{N\geq 0}$ over the realization map $\widehat{\zeta_A}:\Rcal\to\Zcal_\infty$ of Thakur's multiple zeta values. More precisely, we have the following theorem:
\begin{thm}[restated as Theorem \ref{thm:D_N_derivation_rigorous}]\label{thm:main4}
    For $N\geq 0$, let $\Dcal_N : \Rcal \to \Zcal_\infty$  be the unique $\FF_p$-linear map such that
    $\Dcal_N(x_{\varnothing})=\delta_{N,0}$
    and 
    $$\Dcal_N(x_{\ww{s}}) := \sum_{\substack{n_1, \ldots, n_m \ge 0 \\ n_1 + \dots + n_m = N}} \left( \prod_{j=1}^m \binom{-s_j}{n_j} \right) \zeta_A(s_1-n_1(r-1), \dots, s_m-n_m(r-1))$$
    for non-empty index $\ww{s}=(s_1,\ldots,s_m)\in\Ical$. Then
    for any indices $\ww{r}, \ww{s} \in \Ical$, we have $\Dcal_0=\widehat{\zeta_A}$ and
    $$\Dcal_N(x_{\ww{r}} \ast x_{\ww{s}}) = \sum_{k=0}^N \Dcal_k(x_{\ww{r}}) \Dcal_{N-k}(x_{\ww{s}}).$$
\end{thm}

In effect, Theorem \ref{thm:main4} provides a family of explicit relations between Thakur's multiple zeta values at both positive and non-positive indices (cf. \eqref{eq:derivation}, \eqref{eq:N2_final_derivation}).

\medskip
The introduction of $u$-multiple zeta values naturally raises several intriguing questions for future research. For instance, given that multiple harmonic $u$-series serve as a bridge between Thakur's multiple zeta values analytically and finite multiple zeta values over $K$ algebraically, one might ask if they admit a natural connection to \textit{$v$-adic multiple zeta values} (see \cite{CM19, CM21, CCM22}), or even to \textit{adelic multiple zeta values over $K$} (see \cite{Chen23}). The author aims to explore some of these directions in future work.

\medskip
Having stated our main results, we now outline the contents of the remaining sections. 

In \S\ref{section:2}, we define the $u$-bracket and $u$-Carlitz factorial (see Definitions \ref{df-u-bracket} and \ref{df-u-factorial}), and establish a $u$-analog of Legendre's formula in \S\ref{section:2.1} (see Proposition \ref{prop-u-analog-Sinnott}). Furthermore, we provide estimates on the coefficients of the Carlitz module and the $u$-bracket in \S\ref{section:2.2}, which will be essential for the subsequent analysis. 

In \S\ref{section:3}, we investigate the finite multiple $u$-series and $u$-multiple zeta values, whose definitions are given in \S\ref{section:3.1}. In \S\ref{section:3.2}, we work in a more general setting to prove that these values satisfy the same $r$-shuffle relations as Thakur's multiple zeta values at positive indices. We then establish an Euler--Carlitz-type formula for finite multiple harmonic $u$-series at Carlitz torsion points in \S\ref{section:3.3} (see Theorem \ref{thm:finite Euler Carlitz}). In \S\ref{section:3.4}, we provide the proofs of Theorems \ref{thm:main1} and \ref{thm:main2}. In conjunction with the Euler--Carlitz-type formula, these theorems recover the original Euler--Carlitz formula for Thakur's multiple zeta values and imply the vanishing of finite multiple zeta values over $K$ at $r$-even integers (see Corollaries \ref{cor:Euler-Carlitz} and \ref{cor:vanishing of fmzv qeven}). 

In \S\ref{section:4}, we further discuss the properties of $u$-multiple zeta values. In \S\ref{section:4.1}, we view these values as functions on the Drinfeld upper-half plane by setting $u = \exp_\C(\cpi z)$, and prove that they are rigid analytic. In \S\ref{section:4.2}, we treat $u$-multiple zeta values as formal power series in $u$. Specifically, we compute their explicit coefficients and provide a proof of Theorem \ref{thm:main3}. Finally, in \S\ref{section:4.3}, inspired by this formal expansion, we define the operators $(\Dcal_N)_{N \geq 0}$ and prove Theorem \ref{thm:main4} by applying the abstract framework of $r$-shuffle relations established in \S\ref{section:3.2}.
\section*{Funding}
This work was supported by the National Science and Technology Council [113-2628-M-007-004].
\section*{Acknowledgements}
The author would like to thank Chieh-Yu Chang for suggesting this research and for his careful review of the manuscript, which greatly improved its quality. The author also thanks Song-Yun Chen and Fei-Jun Huang for helpful discussions. The author is also grateful to the referee for the careful reading of the manuscript and many valuable suggestions that improved this paper. The financial support provided by the NSTC during the course of this work is also gratefully acknowledged.
\section{\texorpdfstring{$u$}{u}-Bracket and \texorpdfstring{$u$}{u}-Carlitz Factorial}\label{section:2}
\subsection{\texorpdfstring{$u$}{u}-Bracket and \texorpdfstring{$u$}{u}-Carlitz Factorial}\label{section:2.1}
\begin{df}\label{df-u-bracket}
    For each $a \in A$, we define the \textit{$u$-bracket} as the polynomial in the formal variable $u$ given by
    $$
    [a]_u := \frac{\C_a(u)}{u} \in \CC_\infty[u]\subset\CC_\infty\mbb{u}.
    $$
    When $u$ is specialized to a value in $\CC_\infty$, this recovers the evaluation
    $$
    [a]_u = 
    \begin{cases} 
    \frac{\C_a(u)}{u} & \text{if } u \neq 0, 
    \\ a & \text{if } u = 0. 
    \end{cases}
    $$
\end{df}

One sees that $[a]_u \to a$ as $u \to 0$, analogous to the classical $q$-bracket $[n]_q = \frac{q^n-1}{q-1} \to n$ as $q \to 1$. Furthermore, we have the following immediate properties:
\begin{lem}\label{lem-Fq-linear-u-bracket}
    For $a,b\in A$ and $\varepsilon\in\FF_r$, we have
    \begin{enumerate}
        \item $[a+b]_u=[a]_u+[b]_u$.
        \item $[\varepsilon a]_u=\varepsilon[a]_u$.
        \item $[a]_u[b]_{\C_a(u)}=[a]_{\C_b(u)}[b]_u=[ab]_u$.
    \end{enumerate}
\end{lem}
\begin{proof}
    (1), (2) follow directly from the $\FF_r$-linearity of Carlitz action. (3) follows from the fact that
    $$[a]_u[b]_{\C_a(u)}=\frac{\C_a(u)}{u}\cdot\frac{\C_b(\C_a(u))}{\C_a(u)}=\frac{\C_{ab}(u)}{u}=\frac{\C_a(\C_b(u))}{\C_b(u)}\cdot\frac{\C_b(u)}{u}=[a]_{\C_b(u)}[b]_u.$$
\end{proof}

We now use $u$-brackets to define a $u$-analog of Carlitz factorials.
\begin{df}\label{df-u-factorial}
    Let $n=\sum_{d=0}^\infty n_dr^d\in\ZZ_{\geq 0}$ where $0\leq n_d\leq r-1$. We define the \textit{$n$-th $u$-Carlitz factorial} by
    $$
    \Gamma_{u,n+1}=\prod_{d=0}^\infty D_{u,d}^{n_d}\in\CC_\infty[u]
    $$
    where for each $d\in\ZZ_{\geq 0}$,
    $$
    D_{u,d}:=\prod_{a\in A_{+,d}}[a]_u\in\CC_\infty[u].
    $$
\end{df}

Recall that the $n$-th Carlitz factorial is defined by 
$\Gamma_{n+1}:=\prod_{d=0}^\infty D_d^{n_d}$
where
$
D_d=\prod_{a\in A_{+,d}}a\in A
$.
A celebrated result by W. Sinnott states that (see \cite[Theorem 9.1.1]{Goss1996})
\begin{equation}\label{eq-Sinnott's identity}
    \Gamma_{n+1}=\prod_{v \text{ monic irreducible}}v^{\sum_{e\geq 1}\floor{n/|v^e|_\infty}},
\end{equation}
which serves as the function field analog of Legendre's formula for the classical factorial
$$
n!=\prod_{p\text{ prime}}p^{\sum_{e\geq 1}\floor{n/|p^e|}}.
$$

We briefly recall the notion of \textit{Carlitz cyclotomic polynomials} (see \cite[Definition 7.1.4]{Papikian2023}). For $a\in A_+$, the $a$-th Carlitz cyclotomic polynomial is defined as $$\Phi^\C_a(X):=\prod_{\substack{\deg b<\deg a\\\gcd(a,b)=1}}(X-\C_b(\lambda_a))\in K[X],$$
which serves as the function field analog of the classical cyclotomic polynomial $\Phi_k(X)\in\QQ[X]$. Then the following proposition provides a $u$-analog of Sinnott's identity.
\begin{prop}\label{prop-u-analog-Sinnott}
    For $n\in\ZZ_{\geq 0}$, we have
    $$\Gamma_{u,n+1}=\prod_{\substack{a\in A_+\\a\neq 1}}\Phi^{\C}_a(u)^{\floor{n/|a|_\infty}}$$
    where $\Phi^{\C}_a(x)$ denotes the $a$-th Carlitz cyclotomic polynomial.
\end{prop}
\begin{proof}
    First, notice that for $d\in\ZZ_{\geq 0}$, we have
    $$
    D_{u,d}=u^{-r^d}\prod_{b\in A_{+,d}}\C_b(u)=u^{-r^d}\prod_{b\in A_{+,d}}\prod_{\substack{a\in A_+\\a\mid b}}\Phi^{\C}_a(u)=u^{-r^d}\prod_{a\in A_{+,\leq d}}\Phi^{\C}_a(u)^{r^{d-\deg a}}.
    $$
    For the last equality, we observe that each $a \in A_{+,\leq d}$ has exactly $r^{d-\deg a}$ monic multiples of degree $d$ in $A$.
    Now, write $n=\sum_{d=0}^\infty n_dr^d$, $0\leq n_d\leq r-1$. 
    Then
    \begin{equation}\label{eq:gamma u n}
        \Gamma_{u,n+1}
                    = u^{-\sum_{d=0}^\infty n_d r^d} \prod_{d=0}^\infty \prod_{a \in A_{+,\le d}} \Phi^{\C}_a(u)^{n_d r^{d-\deg a}}.
    \end{equation}
   For each $a\in A_+$, observe that the total exponent of $\Phi^{\C}_a(u)$ in \eqref{eq:gamma u n} is
    \[
    \sum_{d \ge \deg a} n_d r^{d-\deg a}=\floor{n/r^{\deg a}}=\floor{n/|a|_\infty}.
    \]
    Finally, note that $\Phi^{\C}_1(u)=u$. 
    We conclude
    \[
    \Gamma_{u,n+1} = u^{-n}\Phi^{\C}_1(u)^n \prod_{\substack{a \in A_+\\a\neq 1}} \Phi^{\C}_a(u)^{\lfloor n / |a|_\infty \rfloor}=\prod_{\substack{a \in A_+\\a\neq 1}} \Phi^{\C}_a(u)^{\lfloor n / |a|_\infty \rfloor}.
    \]
\end{proof}

To see that Proposition \ref{prop-u-analog-Sinnott} serves as a $u$-analog of \eqref{eq-Sinnott's identity}, we require the following lemma:
\begin{lem} \label{lem-cyclotomic-zero}
    Let $a \in A_+$ with $a \neq 1$. Then the $a$-th Carlitz cyclotomic polynomial $\Phi^{\C}_a(u)$ evaluated at $u=0$ satisfies
    \[
    \Phi^{\C}_a(0) = 
    \begin{cases} 
        v, & \text{if } a = v^e \text{ for some monic irreducible polynomial } v \text{ and } e \ge 1, \\
        1, & \text{if } a \text{ has at least two distinct monic irreducible factors.}
    \end{cases}
    \]
\end{lem}

\begin{proof}
    Note that 
    $$
    [a]_u=\frac{\C_a(u)}{u} = u^{-1}\prod_{b \mid a} \Phi^{\C}_b(u)=\prod_{\substack{b\mid a\\b\neq 1}}\Phi^{\C}_b(u).
    $$ 
    Evaluating both sides at $u=0$, we obtain 
    \[
    a = \prod_{\substack{b \mid a \\ b \neq 1}} \Phi^{\C}_b(0)=\prod_{b \mid a} f(b),
    \]
    where $f:A_+\to A$ is defined by $f(1) = 1$ and $f(b) = \Phi^{\C}_b(0)$ for $b \neq 1$.

    By the Möbius inversion formula on the divisibility poset of $A_+$ (see \cite[Section 3.7]{Stanley2012}), we can express $f(a)$ as
    $$f(a) = \prod_{b \mid a} b^{\mu(a/b)},$$
    where $\mu$ is the Möbius function on $A_+$ defined by $\mu(1) = 1$, $\mu(a) = (-1)^k$ if $a$ is a product of $k$ distinct monic irreducible polynomials, and $\mu(a) = 0$ otherwise.
     
    We now determine the values of $f(a)$ by considering the following cases.
    
    \textbf{Case 1: $a = v^e$ for some monic irreducible polynomial $v\in A_+$.} \\
    Any divisors of $a$ are of the form $v^k$ for $0 \le k \le e$ and $\mu(v^e / v^k) = \mu(v^{e-k})$ is non-zero only when $e-k = 0$ or $1$. Thus, we obtain
    \[
    f(v^e) = (v^e)^{\mu(1)} \cdot (v^{e-1})^{\mu(v)} = (v^e)^1 \cdot (v^{e-1})^{-1} = v.
    \]
    Hence, $\Phi^{\C}_{v^e}(0) = v$.

    \textbf{Case 2: $a$ has at least two distinct monic irreducible factors.} \\
We proceed by induction on $\deg a$. The base case is clear by Case 1. Suppose $\Phi^{\C}_b(0) = 1$ for all $b\in A_+$ with $\deg b<\deg a$. Let $a=v_1^{e_1}v_2^{e_2}\cdots v_k^{e_k}\in A_+$ where $v_i$ are distinct monic irreducible polynomials, $e_i\in \NN$, and $k\geq 2$. Then we have 
\[
a = \Phi^{\C}_a(0) \cdot \left( \prod_{i=1}^k \prod_{j=1}^{e_i} \Phi^{\C}_{v_i^j}(0) \right) \cdot \left( \prod_{\substack{b \mid a, \ b \neq a \\ b \text{ has } \ge 2 \text{ irreducible factors}}} \Phi^{\C}_b(0) \right).
\]
By Case 1 and induction hypothesis, we have
\[
a=\Phi^{\C}_a(0)\cdot\prod_{i=1}^k \prod_{j=1}^{e_i} v_i\cdot 1=\Phi_a^{\C}(0)\cdot\prod_{i=1}^k v_i^{e_i}=\Phi_a^{\C}(0)\cdot a.
\]
This forces $\Phi_a^{\C}(0)=1$, which completes the proof.
\end{proof}

Hence, as $u \to 0$, Lemma \ref{lem-cyclotomic-zero} recovers Sinnott's identity \eqref{eq-Sinnott's identity}.
\subsection{Estimates for the \texorpdfstring{$u$}{u}-Bracket}\label{section:2.2}
In this subsection, we characterize the domain of $u$ for which $[a]_u$ diverges as $\deg a \to \infty$. These estimates will be essential for the subsequent analysis.
\begin{df}
    Define $\mathfrak{D}$ to be the subset 
    $$\Dfk=\CC_\infty\setminus\{u\in\CC_\infty:|u|_\infty=r^{k+1/(r-1)}, k\in\ZZ_{\leq 0}\}.$$
\end{df}
\begin{rmk}
    We remark that $\mathfrak{D}$ contains the following regions:
    \begin{enumerate}
        \item $u = 0$.
        \item $\{ u \in \CC_\infty : |u| > r^{1/(r-1)} \}$.
    \end{enumerate}
    Moreover, $\mathfrak{D}$ avoids all non-zero Carlitz torsion points, as the absolute value of any such point is of the form $r^{k + \frac{1}{r-1}}$ for some $k \in \ZZ_{\leq 0}$ (see \cite[Proposition 12.13]{Rosen2002}).
\end{rmk}

We shall prove that for $u \in \mathfrak{D}$, $|[a]_u|_\infty \to \infty$ as $\deg a \to \infty$ by the following estimates on the terms $[a,i]u^{r^i}$ of Carlitz polynomials $\C_a(u)$. 
\begin{lem}\label{lem-i_0(d)}
    For $0\neq u\in\mathfrak{D}$ and $a\in A_{d}$, there exists a unique $0\leq i_0:=i_0(d)\leq d$ such that 
    \[
        |[a,i]u^{r^i}|_\infty < |[a,i_0]u^{r^{i_0}}|_\infty
    \]
    for all $0\leq i\leq d$ with $i\neq i_0$. Moreover, we have the following properties:
    \begin{enumerate}
        \item $i_0$ depends only on $d$ for any fixed $u$.
        \item $i_0(d)\to\infty$ as $d\to\infty$.
        \item $d-i_0(d)=\kappa_u$  for all $d$ large enough, where $\kappa_u$ is a non-negative integer depending only on $u$.
    \end{enumerate}
\end{lem}

\begin{proof}
    Since $u\neq 0$, we may write $|u|_\infty=r^{\eta}$ for some $\eta \in \mathbb{R}$. For $a\in A$ with $\deg a=d$ and each $0\le i\le d$, we have (see \cite[Proposition 12.11]{Rosen2002})
    \[
        |[a,i]|_\infty = r^{r^i(d-i)}.
    \]
    Hence,
    \[
        |[a,i]u^{r^i}|_\infty 
        = r^{r^i(d-i)} r^{r^i\eta}
        = r^{r^i(d-i+\eta)} .
    \]
    Consider the real-valued function
    \[
        R_d(x) = r^{x}(d-x+\eta),\qquad x\in\mathbb{R}.
    \]
    Standard calculus shows that $R_d$ attains its maximum at
    \[
        x_* = d+\eta - \frac{1}{\ln r},
    \]
    and that $R_d(x)$ is strictly increasing on $[0,x_*]$ and strictly decreasing on $[x_*,\infty)$.  

    Let $n=\lfloor x_* \rfloor$. Then the maximum among $\{R_d(i)\}_{0\le i\le d}$ is determined as follows:

    \textbf{Case 1: $n> d-1$.} The maximum is attained at $i=d$, so
    \[
        \max_{0\le i\le d} R_d(i) = R_d(d),
    \]
    and the maximum is unique.

    \textbf{Case 2: $n< 0$.} The maximum is attained at $i=0$, so
    \[
        \max_{0\le i\le d} R_d(i) = R_d(0),
    \]
    and the maximum is unique.
    
    \textbf{Case 3: $0\leq n\le d-1$.} The maximum is 
    \[
        \max_{0\le i\le d} R_d(i) = \max\{R_d(n), R_d(n+1)\}.
    \]
    We now determine when $R_d(n)=R_d(n+1)$.  
    We compute that
    \begin{align*}
        R_d(n+1)-R_d(n)=0 &\iff
        r^{n}\left((r-1)(d-n+\eta)-r\right)=0 \\
        &\iff
        d-n+\eta = \frac{r}{r-1}.
    \end{align*}
    Since $n\leq d-1$, $d-n$ is a positive integer. Thus, this equality fails whenever 
    $\eta\notin \mathbb{Z}_{\le 0}+\frac{1}{r-1}$.  
    Consequently, under the assumption that $u \in \mathfrak{D}$, the set $\{R_d(i)\}_{0\le i\le d}$ has a unique maximum.

    We now analyze this more precisely by considering the value of $\eta$:

    \textbf{Case A: $\eta>\frac{1}{r-1}$.} Note that $$0<\frac{1}{\ln r}-\frac{1}{r-1}<1.$$ One verifies that 
    \[
        n=\lfloor d+\eta - \tfrac{1}{\ln r}\rfloor \ge d-1.
    \]
    If $n=d-1$, then $d-n+\eta = 1+\eta > 1+\frac{1}{r-1} = \frac{r}{r-1}$, so 
    $R_d(n+1)>R_d(n)$ and the unique maximum is $R_d(d)$.  
    If $n>d-1$, the maximum is precisely $R_d(d)$ as shown in Case 1.  
    Therefore,
    \[
        \max_{0\le i\le d} R_d(i)=R_d(d),
    \]
    and hence $i_0(d)=d$ in this case.

    \textbf{Case B: $\eta<\frac{1}{r-1}$ and $\eta\notin \mathbb{Z}_{\le 0}+\frac{1}{r-1}$.} One has 
    \[
        n=\lfloor d+\eta - \tfrac{1}{\ln r}\rfloor \le d-1.
    \]
    We remark that for $d$ large enough, $n\geq 0$. By Case 3, the unique maximum lies in $\{R_d(n),R_d(n+1)\}$. Note that
    \begin{align*}
        R_d(n+1)-R_d(n)>0 &
        \iff
        d-n+\eta > \frac{r}{r-1}.
    \end{align*}
    We compute
    \[
        d-n+\eta=d-\lfloor d+\eta-\tfrac{1}{\ln r}\rfloor+\eta
        =\eta-\lfloor \eta-\tfrac{1}{\ln r}\rfloor,
    \]
    so the choice between $n$ and $n+1$ depends only on $\eta$.  
    Consequently, either
    \[
        i_0(d)=\lfloor d+\eta-\tfrac{1}{\ln r}\rfloor
        = d+\lfloor \eta-\tfrac{1}{\ln r}\rfloor,
    \]
    or
    \[
        i_0(d)=\lfloor d+\eta-\tfrac{1}{\ln r}\rfloor+1
        = d+1+\lfloor \eta-\tfrac{1}{\ln r}\rfloor.
    \]
The three desired properties follow immediately from the above analysis, completing the proof.
\end{proof}

\begin{cor}\label{cor-[a]u-to-infty}
    For $u\in\mathfrak{D}$, we have
    \[
        |[a]_u|_\infty \to\infty
    \]
    as $\deg a\to\infty$.
\end{cor}

\begin{proof}
    For $u=0$, the result is clear. Suppose $u\neq 0$. Write $|u|_\infty=r^{\eta}$ for some $\eta\in\RR$. Let $a\in A$ with $d=\deg a\geq 1$. By Lemma \ref{lem-i_0(d)} and the strong triangle inequality, we have
    \[
        |\C_a(u)|_\infty
        = \left|\sum_{i=0}^d [a,i]u^{r^i}\right|_\infty
        = r^{R_{d}(i_0)}\geq r^{R_d(1)}=r^{\,r(d-1+\eta)}.
    \]
    Hence,
    \[
        |[a]_u|_\infty
        = \left|\frac{\C_a(u)}{u}\right|_\infty
        \ge r^{\,r(d-1+\eta)-\eta}
        \to \infty
    \]
    as $d\to\infty$.
\end{proof}
\section{Finite Multiple Harmonic \texorpdfstring{$u$}{u}-Series}\label{section:3}
In this section, we study the finite multiple harmonic $u$-series, which is viewed as a $u$-analog of the truncation of Thakur's multiple zeta values. Specifically, we will first show that they satisfy the $r$-shuffle relations and possess an Euler-Carlitz-type formula. Furthermore, we examine both the ``algebraic'' and ``analytic'' limits of the finite multiple harmonic $u$-series evaluated at Carlitz torsion points, in view of the work of Bachmann, Takeyama, and Tasaka \cite{BTT18}.
\subsection{Finite Multiple Harmonic \texorpdfstring{$u$}{u}-Series}\label{section:3.1}
First, we recall the notations
$$\Ical^{\ext}=\bigcup_{m=1}^\infty \ZZ^m \cup \{\varnothing\}\quad\text{and}\quad\Ical = \bigcup_{m=1}^\infty \NN^m \cup \{\varnothing\}.$$
Also, by convention, an empty sum is defined to be zero and an empty product to be one.

We now introduce $u$-analogs of Thakur's power sums and finite multiple harmonic series.
\begin{df}
    For $s,d\in\ZZ$, we define the
    \textit{$u$-power sum} by
    $$H_d(s;u):=\sum_{a\in A_{+,d}}\frac{1}{[a]_u^s}\in\CC_\infty\mbb{u}.$$
    For any non-empty index $\ww{s}=(s_1,\ldots,s_m)\in\Ical^{\ext}$ and $d\in\ZZ$, we define the
    \textit{$u$-multiple power sum} by
    \begin{align*}
        H_d(\ww{s};u)&:=\sum_{d=d_1>\cdots>d_m\geq 0}H_{d_1}(s_1;u)\cdots H_{d_m}(s_m;u)\\
        &=\sum_{\substack{a_1,\ldots,a_m\in A_+\\d=\deg a_1>\ldots>\deg a_m\geq 0}} \frac{1}{[a_1]_u^{s_1}\cdots[a_m]_u^{s_m}}\in\CC_\infty\mbb{u}.
    \end{align*}
    We conventionally put $H_d(\varnothing;u)=\delta_{d,0}$ where $\delta_{i,j}$ denotes the Kronecker delta. Evaluating $H_{d}(\ww{s};u)$ at any $u\in\CC_\infty$ such that $\C_a(u)\neq 0$ for all $a\in A_{+,\leq d}$ yields a well-defined element in $\CC_\infty$. 
\end{df}

\begin{df}\label{df:fmhus}
    For any non-empty index $\ww{s}=(s_1,\ldots,s_m)\in\Ical^{\ext}$ and $d\in\ZZ$, we define the \textit{finite multiple harmonic $u$-series} by
    \begin{align*}
        H_{<d}(\ww{s};u)&:=\sum_{d>d_1>\cdots>d_m\geq 0}H_{d_1}(s_1;u)\cdots H_{d_m}(s_m;u)\\&=
        \sum_{\substack{a_1,\ldots,a_m\in A_+\\d>\deg a_1>\cdots>\deg a_m\geq 0}}\frac{1}{[a_1]_{u}^{s_1}\cdots [a_m]_{u}^{s_m}}\in\CC_\infty\mbb{u}.
    \end{align*}
    We conventionally put $H_{<d}(\varnothing;u)=1$. Evaluating $H_{<d}(\ww{s};u)$ at any $u\in\CC_\infty$ such that $\C_a(u)\neq 0$ for all $a\in A_{+,< d}$ yields a well-defined element in $\CC_\infty$.
\end{df}

Note that whenever $\dep(\ww{s}) > d+1$, $H_{d}(\ww{s};u)$ is an empty sum and thus equals zero. Similarly, $H_{<d}(\ww{s};u)=0$ whenever $\dep(\ww{s}) > d$.

\begin{rmk}One observes that as $u \to 0$, the $u$-multiple power sums and finite multiple harmonic $u$-series reduce to Thakur's multiple power sums and finite multiple harmonic series, respectively. That is, for any index $\ww{s} = (s_1, \dots, s_m) \in \Ical^{\ext}$,$$\lim_{u \to 0} H_d(\ww{s}; u) = S_d(\ww{s}) = \sum_{\substack{a_1, \ldots, a_m \in A_+ \\ d = \deg a_1 > \cdots > \deg a_m \ge 0}} \frac{1}{a_1^{s_1} \cdots a_m^{s_m}}$$and
\[
    \lim_{u \to 0} H_{<d}(\ww{s}; u) = S_{<d}(\ww{s}) := \sum_{\substack{a_1, \ldots, a_m \in A_+ \\ d > \deg a_1 > \cdots > \deg a_m \ge 0}} \frac{1}{a_1^{s_1} \cdots a_m^{s_m}}.
\]
Hence, we have
$\lim_{d\to\infty}\lim_{u\to 0}H_{<d}(\ww{s};u)=\zeta_A(\ww{s}).$
\end{rmk}

Parallel to the classical theory, we introduce a $u$-analog of Thakur's multiple zeta values.
\begin{df}\label{df:uMZV}
    For any non-empty index $\ww{s}=(s_1,\ldots,s_m)\in\Ical^{\ext}$ and $u\in\CC_\infty$ with $\C_a(u)\neq 0$ for all $a\in A_+$, we define the \textit{$u$-multiple zeta value} by
    $$
    \zeta_u(\ww{s}):=\lim_{d\to\infty}H_{<d}(\ww{s};u)
    $$
    whenever the series converges in $\CC_\infty$.
\end{df}
\begin{rmk}
    Corollary \ref{cor-[a]u-to-infty} implies that for $u\in\Dfk$, the series $\zeta_u(\ww{s})$ converges for any positive index $\ww{s}\in\Ical$. For non-positive indices, note that for each fixed $s\leq 0$, the $u$-power sum $H_d(s;u)$ vanishes for all sufficiently large $d$ (see \cite[Theorem 5.1.2]{Thakur2004}). Then Corollary \ref{cor-[a]u-to-infty} together with this observation guarantee the convergence of $\zeta_u(\ww{s})$.
\end{rmk}
\begin{rmk}
    In general, $\zeta_u(\ww{s})$ may not converge to $\zeta_A(\ww{s})$ as $u\to 0$. However, one can check by using similar arguments in Theorem \ref{thm:limit of u} that there exists a sequence $(u_n)_{n\geq 1}$ tending to $0$ such that $\lim_{n\to\infty}\zeta_{u_n}(\ww{s})=\zeta_A(s)$ for any index $\ww{s}\in\Ical^{\ext}$.
\end{rmk}

\subsection{\texorpdfstring{$r$}{r}-Shuffle Relations}\label{section:3.2}

In this subsection, we aim to show that both finite multiple harmonic $u$-series and $u$-multiple zeta values at positive indices satisfy the $r$-shuffle relations. To this end, we first establish the result in a more formal and general framework by considering the following setting: 

Let $\mathbb{K}$ be an integral domain containing $\FF_r$ and fix $d_0\in\NN$. For each \(a\in A_{<d_0}\), we assign an element \([a]\in \mathbb{K}\).
We assume the bracket satisfying the following properties:
\begin{enumerate}
    \item For all $a\in A_{+,<d_0}$, $[a]\in\KK^\times$.
    \item For all \(a,b\in A_{<d_0}\),
    $$[a+b]=[a]+[b].$$
    \item For all $a\in A_{<d_0}$ and $\varepsilon\in\FF_r$,
    $$[\varepsilon a]=\varepsilon [a].$$
\end{enumerate}
In particular, one sees that for $u\in\CC_\infty$ with $\C_a(u)\neq 0$ for all $a\in A_{+,<d_0}$, the $u$-bracket $[a]_u$ satisfies all the properties.

Next, for any index $\ww{s}=(s_1,\ldots,s_m)\in\Ical$, we consider the formal version of multiple power sum
\begin{align}\label{eq:H_d def}
    \H_{d}(\ww{s}) = \sum_{\substack{a_1,\ldots,a_m\in A_+ \\ d=\deg a_1 > \cdots > \deg a_m \geq 0}} \frac{1}{[a_1]^{s_1}\cdots[a_m]^{s_m}} \in \KK,\quad 0\leq d<d_0,
\end{align}
and the corresponding finite multiple harmonic series
\begin{align}\label{eq:H_<d def}
        \H_{<d}(\ww{s})&:=\sum_{\substack{a_1,\ldots,a_m\in A_+\\d>\deg a_1>\cdots>\deg a_m\geq 0}}\frac{1}{[a_1]^{s_1}\cdots [a_m]^{s_m}}\in\KK,\quad 1\leq d\leq d_0.
    \end{align}
By convention, we put $\H_{<0}(\ww{s})=0$ for all non-empty $\ww{s}\in\Ical$ and $\H_{<d}(\varnothing)=1$ for all $0\leq d\leq d_0$.
We let
$$\widehat{\H_{<d}}:\Rcal\to\KK$$
be the unique $\FF_p$-linear map such that
$$1=x_{\varnothing}\mapsto 1,\quad x_{\ww{s}}\mapsto \H_{<d}(\ww{s}),$$
and recall the notation $\ww{s}^- \in \Ical$, denoting  the index obtained by removing the first entry of $\ww{s}\in\Ical$.
Then we show that $\H_{<d_0}(\ww{s})$ satisfies the $r$-shuffle relations.
For this purpose, we establish several preliminary lemmas:
\begin{lem}\label{lem:H<d=sumHd}
    Let $0\leq d<d_0$ and $\ww{s}=(s_1,\ldots,s_m)\in\Ical$ be a non-empty index. Then we have
    $$\H_{<d}(\ww{s})=\sum_{0\leq d_1<d}\H_{d_1}(\ww{s}).$$
\end{lem}

\begin{lem}\label{lem:d=d product <d}
     Let $0\leq d<d_0$ and $\ww{s}=(s_1,\ldots,s_m)\in\Ical$ be a non-empty index.  Then we have
    $$\H_{d}(\ww{s})=\H_{d}(s_1)\H_{<d}(\ww{s}^-).$$
\end{lem}
In the following proofs of Lemma \ref{lem:shuffle base} and Theorem \ref{thm:abstract-r-shuffle}, we closely follow the arguments of Shi \cite[Theorem 3.1.4]{Shi2018} for $S_{<d}(\ww{s})$ (see also \cite{Thakur2010}).
\begin{lem}\label{lem:shuffle base}
     Let $0\leq d< d_0$ and $r_1,s_1\in\NN$. Then we have
    $$\H_{d}(r_1)\H_{d}(s_1)=\H_{d}(r_1+s_1)+\sum_{\substack{i,j\geq 1\\i+j=r_1+s_1}}\Delta_{r_1,s_1}^{i,j} \H_{d}(i,j)$$
    where
    $\Delta_{r_1,s_1}^{i,j}$ is defined in \eqref{eq:df_Delta}.
\end{lem}
\begin{proof}
    Note that 
    \begin{align*}
        \H_{d}(r_1)\H_{d}(s_1)&= \left(\sum_{a\in A_{+,d}}\frac{1}{[a]^{r_1}}\right)\left(\sum_{b\in A_{+,d}}\frac{1}{[b]^{s_1}}\right)
        \\
        &= \sum_{a\in A_{+,d}}\frac{1}{[a]^{r_1+s_1}}+\sum_{\substack{a,b\in A_{+,d}\\a\neq b}}\frac{1}{[a]^{r_1}[b]^{s_1}}
        \\
        &= \H_{d}(r_1+s_1)+\sum_{\substack{a,b\in A_{+,d}\\a\neq b}}\frac{1}{[a]^{r_1}[b]^{s_1}}.
    \end{align*}
    Next, recall the partial fraction decomposition: For elements $x\neq y$ in an integral domain, we have
    $$\frac{1}{x^{r_1}y^{s_1}}=\sum_{i+j=r_1+s_1}\frac{1}{(x-y)^j}\left(\frac{(-1)^{s_1}\binom{j-1}{s_1-1}}{x^i}+\frac{(-1)^{j-r_1}\binom{j-1}{r_1-1}}{y^i}\right).$$
    Thus, we have
    \begin{align}
        \notag\sum_{\substack{a,b\in A_{+,d}\\a\neq b}}\frac{1}{[a]^{r_1}[b]^{s_1}}&=\sum_{\substack{a,b\in A_{+,d}\\a\neq b}}\sum_{i+j=r_1+s_1}\frac{1}{\left([a]-[b]\right)^j}\left(\frac{(-1)^{s_1}\binom{j-1}{s_1-1}}{[a]^i}+\frac{(-1)^{j-r_1}\binom{j-1}{r_1-1}}{[b]^i}\right)
        \\\label{eq:shuffle1}&=\sum_{\substack{a\in A_{+,d}\\ 0\neq f\in A_{<d}}}\sum_{i+j=r_1+s_1}\frac{1}{[f]^j}\left(\frac{(-1)^{s_1}\binom{j-1}{s_1-1}}{[a]^i}+\frac{(-1)^{j-r_1}\binom{j-1}{r_1-1}}{[a-f]^i}\right).
    \end{align}
    Note that the first summand in \eqref{eq:shuffle1} equals 
    \begin{align*}
        \sum_{i+j=r_1+s_1}\sum_{\substack{a\in A_{+}, 0\neq f\in A\\ d=\deg a>\deg f}}\frac{(-1)^{s_1}\binom{j-1}{s_1-1}}{[a]^i[f]^j}&=\sum_{i+j=r_1+s_1}\left(\sum_{\varepsilon\in\FF_r^\times}\frac{1}{\varepsilon^j}\right)\left(\sum_{\substack{a,f\in A_{+}\\ d=\deg a>\deg f}}\frac{(-1)^{s_1}\binom{j-1}{s_1-1}}{[a]^i[f]^j}\right)
        \\&=\sum_{\substack{i+j=r_1+s_1\\(r-1)\mid j}}(-1)^{s_1+1}\binom{j-1}{s_1-1} \H_{d}(i,j).
    \end{align*}
    Similarly, the second summand in \eqref{eq:shuffle1} becomes
    \begin{align*}
        \sum_{i+j=r_1+s_1}\sum_{\substack{(a-f)\in A_{+}, 0\neq f\in A\\d=\deg (a-f)>\deg f}}\frac{(-1)^{j-r_1}\binom{j-1}{r_1-1}}{[a-f]^i[f]^j}&=\sum_{\substack{i+j=r_1+s_1\\(r-1)\mid j}}(-1)^{j-r_1+1}\binom{j-1}{r_1-1} \H_{d}(i,j).
    \end{align*}
    Note that $(-1)^{s_1+1}=(-1)^{s_1-1}$ and $(-1)^{j-r_1+1}=(-1)^{r_1-1}$ since $(r-1)\mid j$. It follows that 
    $$\H_{d}(r_1)\H_{d}(s_1)=\H_{d}(r_1+s_1)+\sum_{\substack{i,j\geq 1\\i+j=r_1+s_1}}\Delta_{r_1,s_1}^{i,j} \H_{d}(i,j).$$
\end{proof}

With the necessary lemmas in hand, we prove that $\widehat{\H_{<d_0}}:\Rcal\to\KK$ forms an $\FF_p$-algebra homomorphism. That is, ${\H_{<d_0}}(\ww{s})$ satisfies the $r$-shuffle relations.
\begin{thm}\label{thm:abstract-r-shuffle}
    Let $1\leq d\leq d_0$. Then $\widehat{\H_{<d}}$ is an $\FF_p$-algebra homomorphism. That is, 
    $$\widehat{\H_{<d}}(x_{\ww{r}}\ast x_{\ww{s}})=\widehat{\H_{<d}}(x_{\ww{r}})\widehat{\H_{<d}}(x_{\ww{s}})=\H_{<d}(\ww{r})\H_{<d}(\ww{s})$$
    for any indices $\ww{r},\ww{s}\in\Ical$.
\end{thm}
\begin{proof}
     We proceed by induction on the total depth $k=\dep(\ww{r})+\dep(\ww{s})$. When $\ww{r}$ or $\ww{s}$ is empty, the results is clear. In particular, the base case $k=1$ holds. Suppose the result holds for all $k<\ell$. We may assume both $\ww{r}$ and $\ww{s}$ are non-empty.  Let 
    $\ww{r}=(r_1,\ldots,r_n),\ww{s}=(s_1,\ldots,s_m)\in\Ical$ be non-empty indices
    with $m+n=\ell$.
    Then by Lemma \ref{lem:H<d=sumHd}, we have
    \begin{align*}
        &\widehat{\H_{<d}}(x_{\ww{r}})\widehat{\H_{<d}}(x_{\ww{s}})\\
        &=\left(\sum_{0\leq d_1<d}\H_{d_1}(\ww{r})\right)\left(\sum_{0\leq d_2<d}\H_{d_2}(\ww{s})\right)
        \\
        ={} &\begin{multlined}[t]
        \sum_{0\leq d_2<d_1<d}\H_{d_1}(\ww{r})\H_{d_2}(\ww{s})+\sum_{0\leq d_1<d_2<d}\H_{d_1}(\ww{r})\H_{d_2}(\ww{s})+\sum_{0\leq d_2=d_1<d}\H_{d_1}(\ww{r})\H_{d_2}(\ww{s})\\
        =\sum_{0\leq d_1<d}\H_{d_1}(\ww{r})\left(\sum_{0\leq d_2<d_1}\H_{d_2}(\ww{s})\right)+\sum_{0\leq d_2<d}\left(\sum_{0\leq d_1<d_2}\H_{d_1}(\ww{r})\right)\H_{d_2}(\ww{s})  
            \\+ \sum_{0\leq d_3<d}\H_{d_3}(\ww{r})\H_{d_3}(\ww{s})
        \end{multlined}\\
        &=\sum_{0\leq d_1<d}\H_{d_1}(\ww{r})\H_{<d_1}(\ww{s})+\sum_{0\leq d_2<d}\H_{<d_2}(\ww{r})\H_{d_2}(\ww{s})  
            + \sum_{0\leq d_3<d}\H_{d_3}(\ww{r})\H_{d_3}(\ww{s})
    \end{align*}
    For the third equality, note that empty sums are taken to be zero.
    Hence, by Lemma \ref{lem:d=d product <d}, we obtain
    \begin{multline}\label{eq:shuffle2}
\widehat{\H_{<d}}(x_{\ww{r}})\widehat{\H_{<d}}(x_{\ww{s}})
= \sum_{0\leq d_1<d} \H_{d_1}(r_1) \H_{<d_1}(\ww{r}^-) \H_{<d_1}(\ww{s}) \\
+ \sum_{0\leq d_2<d} \H_{d_2}(s_1) \H_{<d_2}(\ww{s}^-) \H_{<d_2}(\ww{r}) \\
+ \sum_{0\leq d_3<d} \H_{d_3}(s_1) \H_{d_3}(r_1) \H_{<d_3}(\ww{r}^-) \H_{<d_3}(\ww{s}^-).
\end{multline}
By induction hypothesis and Lemma \ref{lem:d=d product <d}, the first summand in \eqref{eq:shuffle2} becomes
    \begin{align*}
        \sum_{0\leq d_1<d}\H_{d_1}({r_1})\H_{<d_1}({\ww{r}^-})\H_{<d_1}(\ww{s})&=\sum_{0\leq d_1<d}\H_{d_1}(r_1)\widehat{\H_{<d_1}}(x_{\ww{r}^-}\ast x_{\ww{s}})\\
        &=
        \sum_{0\leq d_1<d}\widehat{\H_{d_1}}(x_{r_1}(x_{\ww{r}^-}\ast x_{\ww{s}}))
        \\
        &=\widehat{\H_{<d}}(x_{r_1}(x_{\ww{r}^-}\ast x_{\ww{s}})).
    \end{align*}
The second equality can be checked immediately by writing $x_{\ww{r}^-}\ast x_\ww{s}=\sum_{i}c_ix_{\ww{t}_i}$, $c_i\in\FF_p$ and using the $\FF_p$-linearity. Similarly, the second summand in \eqref{eq:shuffle2} equals 
    $$\sum_{0\leq d_2<d} \H_{d_2}(s_1) \H_{<d_2}(\ww{s}^-) \H_{<d_2}(\ww{r})=\widehat{\H_{<d}}(x_{s_1}(x_{\ww{r}}\ast x_{\ww{s}^-})).$$
Finally, by Lemmas \ref{lem:H<d=sumHd}, \ref{lem:d=d product <d}, \ref{lem:shuffle base} and induction hypothesis, the third summand in \eqref{eq:shuffle2} is
\begin{align*}
    &\sum_{0\leq d_3<d} \H_{d_3}(s_1) \H_{d_3}(r_1) \H_{<d_3}(\ww{r}^-) \H_{<d_3}(\ww{s}^-)\\
    &=\sum_{0\leq d_3<d}\H_{d_3}(s_1) \H_{d_3}(r_1) \widehat{\H_{<d_3}}(x_{\ww{r}^-}\ast x_{\ww{s}^-}) \\
    &=
    \sum_{0\leq d_3<d}\left(\H_{d_3}(r_1+s_1)+\sum_{\substack{i,j\geq 1\\i+j=r_1+s_1}}\Delta_{r_1,s_1}^{i,j} \H_{d_3}(i,j)\right)\widehat{\H_{<d_3}}(x_{\ww{r}^-}\ast x_{\ww{s}^-})\\
    &=
    \sum_{0\leq d_3<d}\H_{d_3}(r_1+s_1)\widehat{\H_{<d_3}}(x_{\ww{r}^-}\ast x_{\ww{s}^-}) +\sum_{0\leq d_3<d}\sum_{\substack{i,j\geq 1\\i+j=r_1+s_1}}\Delta_{r_1,s_1}^{i,j} \H_{d_3}(i,j)\widehat{\H_{<d_3}}(x_{\ww{r}^-}\ast x_{\ww{s}^-}) \\
    &=
    \widehat{\H_{<d}}(x_{r_1+s_1}(x_{\ww{r}^-}\ast x_{\ww{s}^-})) +\sum_{\substack{i,j\geq 1\\i+j=r_1+s_1}}\Delta_{r_1,s_1}^{i,j}\sum_{0\leq d_3<d}\H_{d_3}(i)\widehat{\H_{<d_3}}(x_{\ww{r}^-}\ast x_{\ww{s}^-}) \H_{<d_3}(j)\\
    &=
    \widehat{\H_{<d}}(x_{r_1+s_1}(x_{\ww{r}^-}\ast x_{\ww{s}^-})) +\sum_{\substack{i,j\geq 1\\i+j=r_1+s_1}}\Delta_{r_1,s_1}^{i,j}\sum_{0\leq d_3<d}\H_{d_3}(i)\widehat{\H_{<d_3}}(x_{\ww{r}^-}\ast x_{\ww{s}^-}) \widehat{\H_{<d_3}}(x_j)\\
    &=
    \widehat{\H_{<d}}(x_{r_1+s_1}(x_{\ww{r}^-}\ast x_{\ww{s}^-})) +\sum_{\substack{i,j\geq 1\\i+j=r_1+s_1}}\Delta_{r_1,s_1}^{i,j}\sum_{0\leq d_3<d}\H_{d_3}(i)\widehat{\H_{<d_3}}((x_{\ww{r}^-}\ast x_{\ww{s}^-})\ast x_j)\\
    &=
    \widehat{\H_{<d}}(x_{r_1+s_1}(x_{\ww{r}^-}\ast x_{\ww{s}^-})) +\sum_{\substack{i,j\geq 1\\i+j=r_1+s_1}}\Delta_{r_1,s_1}^{i,j}\widehat{\H_{<d}}(x_i((x_{\ww{r}^-}\ast x_{\ww{s}^-})\ast x_j)).
\end{align*}
This completes the proof.
\end{proof}

We now further assume that $\KK$ is a Hausdorff topological ring. Moreover, each $a\in A$ corresponds to an element $[a]\in\KK$ satisfying the following properties:
\begin{enumerate}
    \item For all $a\in A_{+}$, $[a]\in\KK^\times$.
    \item For all \(a,b\in A\),
    $$[a+b]=[a]+[b].$$
    \item For all $a\in A$ and $\varepsilon\in\FF_r$,
    $$[\varepsilon a]=\varepsilon [a].$$
    \item For any index $\ww{s}\in\Ical$, the limit
    $$\mathbf{Z}(\ww{s})=\lim_{d\to\infty}\H_{<d}(\ww{s})$$
    converges in $\KK$.
\end{enumerate}
Then we obtain the following corollary. 
\begin{cor}\label{cor:Z hat shuffle}
    The unique $\FF_p$-linear map
$$\widehat{\mathbf{Z}}:\Rcal\to\KK$$
defined by
$$1=x_{\varnothing}\mapsto 1,\quad x_{\ww{s}}\mapsto \mathbf{Z}(\ww{s}),$$ is an $\FF_p$-algebra homomorphism. That is, 
    $$\widehat{\mathbf{Z}}(x_{\ww{r}}\ast x_{\ww{s}})=\widehat{\mathbf{Z}}(x_{\ww{r}})\widehat{\mathbf{Z}}(x_{\ww{s}})=\mathbf{Z}(\ww{r})\mathbf{Z}(\ww{s})$$
    for any indices $\ww{r},\ww{s}\in\Ical$.
\end{cor} 

Then, we return to the case of finite multiple harmonic $u$-series and $u$-multiple zeta values. For any $d \in \NN$ and $u \in \CC_\infty$ such that $\C_a(u) \neq 0$ for all $a \in A_{+, <d}$, we define$$\widehat{H_{<d}}(\bullet; u) : \Rcal \to \CC_\infty$$as the unique $\FF_p$-linear map such that $\widehat{H_{<d}}(x_{\ww{s}}; u) := H_{<d}(\ww{s}; u).$
Additionally, for $u \in \CC_\infty$ such that $\zeta_u(\ww{s})$ is defined for every index $\ww{s} \in \Ical$, we define the realization map of $u$-multiple zeta values by $\widehat{\zeta_u}:=\lim_{d\to\infty}\widehat{H_{<d}}(\bullet;u)$.
As immediate consequences of Theorem \ref{thm:abstract-r-shuffle} and Corollary \ref{cor:Z hat shuffle}, we obtain the following results.

\begin{cor}\label{cor:Hu shuffle}
    For any $d\in\NN$ and $u\in \CC_\infty$ such that $\C_a(u)\neq 0$ for all $a\in A_{+,<d}$, the map $\widehat{H_{<d}}(\bullet;u)$ is an $\FF_p$-algebra homomorphism. That is,
    $$\widehat{H_{<d}}(x_{\ww{r}}\ast x_{\ww{s}};u) =\widehat{H_{<d}}(x_{\ww{r}};u)\widehat{H_{<d}}(x_{\ww{s}};u)=H_{<d}(\ww{r};u)H_{<d}(\ww{s};u)$$
    for any indices $\ww{r},\ww{s}\in\Ical$.
\end{cor}

\begin{cor}\label{cor:zeta u shuffle}
    For $u\in\CC_\infty$ such that $\zeta_u(\ww{s})$ is defined for every index $\ww{s} \in \Ical$, the map $\widehat{\zeta_u}$ is an $\FF_p$-algebra homomorphism. That is,
    $$\widehat{\zeta_u}(x_{\ww{r}}\ast x_{\ww{s}};u) =\widehat{\zeta_u}(x_{\ww{r}};u)\widehat{\zeta_u}(x_{\ww{s}};u)=\zeta_u(\ww{r};u)\zeta_u(\ww{s};u)$$
    for any indices $\ww{r},\ww{s}\in\Ical$.
\end{cor}
\subsection{Euler--Carlitz-type Formula for \texorpdfstring{$H_{<d}$}{H<d} at Carlitz Torsion Points}\label{section:3.3}
Recall that for each $\nfk\in A_+$,  $\lambda_{\nfk}=\exp_{\C}(\frac{\cpi}{\nfk})$ is the generator of $\Lambda_\nfk$, the $A$-module of Carlitz $\nfk$-torsion points. In this subsection, we establish an Euler--Carlitz-type formula for $H_{<{\deg\nfk}}(s;\lambda_\nfk)$, which we call the \textit{finite Euler--Carlitz formula}. 

We first define the degenerate Bernoulli--Carlitz numbers.
\begin{df}
    For non-zero $a\in A$, we define the \textit{degenerate Bernoulli-Carlitz number} $\dBC_n(a)\in K$ by the identity
    $$\frac{X}{\C_a\left(\frac{X}{a}\right)}=\sum_{n=0}^\infty \frac{\dBC_n(a)}{\Gamma_{n+1}}X^n.$$
\end{df}
This can be viewed as the function field analog of the \textit{degenerate Bernoulli number} $\beta_n(k^{-1})$ defined by Carlitz \cite{Carlitz1956}.
Recall that the Bernoulli--Carlitz number $\BC_n\in K$ is defined by the identity
$$\frac{X}{\exp_{\C}(X)}=\sum_{n=0}^\infty \frac{\BC_n}{\Gamma_{n+1}}X^n.$$
One checks that as $\deg a\to\infty$, $\dBC_n(a)\to \BC_n$ for each $n\in\NN$ since $\C_a(\frac{X}{a})\to \exp_{\C}(X)$.

Then the finite Euler--Carlitz formula is given as follows:
\begin{thm}\label{thm:finite Euler Carlitz}
    Let $s\in\NN$ be a $r$-even integer, i.e., $(r-1)\mid s$. Let $\nfk\in A_+$. Then we have
    \[
        \frac{H_{<\deg\nfk}(s;\lambda_\nfk)}{(\nfk\lambda_\nfk)^s}=\frac{\dBC_s(\nfk)}{\Gamma_{s+1}}.
    \]
    In particular, $H_{<d}(s;\lambda_\nfk)\in K\cdot \lambda_\nfk^s$.
\end{thm}
\begin{proof}
    Let $d=\deg\nfk$. Note that
    $$\C_{\nfk}(X)=\prod_{a\in A_{<d}}\left(X-\C_a(\lambda_\nfk)\right)\quad\text{and}\quad\C_\nfk(X)'=\sum_{a\in A_{<d}}\prod_{\substack{b\in A_{<d}\\ b\neq a}}\left(X-\C_b(\lambda_\nfk)\right).$$
    Thus, the logarithmic derivative of $\C_\nfk(X)$ is
\begin{align*}
\frac{\C_\nfk(X)'}{\C_\nfk(X)}
&= \frac{1}{X}
  + \sum_{0\neq a\in A_{<d}}
    \frac{1}{X-\C_a(\lambda_\nfk)}
\\
&= \frac{1}{X}
  + \sum_{0\neq a\in A_{<d}}
    \frac{1}{X+\C_a(\lambda_\nfk)}
\\
&= \frac{1}{X}
  + \sum_{0\neq a\in A_{<d}}
    \frac{1}{\C_a(\lambda_\nfk)}
    \cdot
    \frac{1}{\displaystyle 1+\frac{X}{\C_a(\lambda_\nfk)}}
\\
&= \frac{1}{X}
    +\sum_{0\neq a\in A_{<d}}
    \frac{1}{\C_a(\lambda_\nfk)}\sum_{s=0}^\infty \frac{(-1)^s}{\C_a(\lambda_\nfk)^s}X^s
\\
&= \frac{1}{X}
    +\sum_{s=0}^\infty \sum_{0\neq a\in A_{<d}}
    \frac{(-1)^s}{\C_a(\lambda_\nfk)^{s+1}}X^s.
\end{align*}
By change of variables, we obtain
\begin{align*}
\frac{X\C_\nfk(\frac{X}{\nfk})'}{\C_\nfk\left(\frac{X}{\nfk}\right)}
&= 1
    +\sum_{s=1}^\infty \sum_{0\neq a\in A_{<d}}\frac{(-1)^{s-1}}{\left(\nfk\C_a(\lambda_\nfk)\right)^s}X^s
\\
&= 1
    +\sum_{s=1}^\infty \left(\sum_{\varepsilon\in\FF_r^\times}\frac{1}{\varepsilon^s}\right)\left(\sum_{ a\in A_{+,<d}}\frac{(-1)^{s-1}}{\left(\nfk\C_a(\lambda_\nfk)\right)^s}\right)X^s\\
&= 1
    +\sum_{\substack{s=1\\(r-1)\mid s}}^\infty \sum_{a\in A_{+,<d}}\frac{(-1)^{s}}{\left(\nfk\C_a(\lambda_\nfk)\right)^s}X^s
\\
&= 1
    +\sum_{\substack{s=1\\(r-1)\mid s}}^\infty (-1)^s \frac{H_{<d}(s;\lambda_\nfk)}{(\nfk\lambda_\nfk)^s} X^s
\\
&= 1
    +\sum_{\substack{s=1\\(r-1)\mid s}}^\infty \frac{H_{<d}(s;\lambda_\nfk)}{(\nfk\lambda_\nfk)^s} X^s.
\end{align*}
For the last equality, we notice that  $(-1)^s=1$ either when $p=2$ or when $p$ is odd with $(r-1)\mid s$. Finally, observe that
$$\frac{X\C_\nfk(\frac{X}{\nfk})'}{\C_\nfk\left(\frac{X}{\nfk}\right)}=\frac{X}{\C_\nfk\left(\frac{X}{\nfk}\right)}=\sum_{s=0}^\infty \frac{\dBC_s(\nfk)}{\Gamma_{s+1}}X^s,$$ so the result follows by comparing the coefficients.
\end{proof}

\subsection{``Limits'' of \texorpdfstring{$H_{<d}$}{H<d} at Carlitz Torsion Points}\label{section:3.4}
We then investigate both ``analytic'' and ``algebraic'' limits of $H_{<\deg \nfk}(\ww{s};\lambda_\nfk)$ as $\deg\nfk\to\infty$ in view of \cite{BTT18}. First, we require the following two lemmas.

\begin{lem}\label{lem:prod-identity}
For $m\ge1$ and elements $X_1,\dots,X_m$, $Y_1,\dots,Y_m$ in a commutative ring, one has
\[
\prod_{i=1}^m X_i-\prod_{i=1}^m Y_i
=\sum_{k=1}^m
\left(\prod_{i<k}Y_i\right)
\left(X_k-Y_k\right)
\left(\prod_{i>k}X_i\right).
\]
\end{lem}

\begin{proof}
We first recall the convention that an empty product equals one. We define for $1\le k\le m+1$
\[
Q_k:=\left(\prod_{i<k}Y_i\right)\left(\prod_{i\ge k}X_i\right).
\]
Then
\[
\prod_{i=1}^m X_i-\prod_{i=1}^m Y_i = Q_1-Q_{m+1} = \sum_{k=1}^m (Q_k-Q_{k+1}).
\]
On the other hand, notice that
\[
Q_k = \left(\prod_{i<k}Y_i\right)X_k\left(\prod_{i>k}X_i\right)\quad\text{and}\quad
Q_{k+1} = \left(\prod_{i<k}Y_i\right)Y_k\left(\prod_{i>k}X_i\right).
\]
Hence,
\[
Q_k-Q_{k+1}
= \left(\prod_{i<k}Y_i\right)\left(\prod_{i>k}X_i\right)(X_k-Y_k),
\]
and summing over $k=1,\dots,m$ yields the claimed identity.
\end{proof}

\begin{lem}\label{lem:[a]n=a}Let $\nfk\in A$. Then for all $a\in A_{<{\deg\nfk}}$,$$|[a]_{\lambda_{\nfk}}-a|_\infty<|a|_\infty \quad\text{and}\quad|[a]_{\lambda_\nfk}|_\infty=|a|_\infty.$$\end{lem}\begin{proof}Let $d=\deg\nfk$. Recall that we have $$|\lambda_\nfk|_\infty=r^{-d+1+\frac{1}{r-1}}\quad\text{and}\quad\left|[a,i]\right|_\infty=r^{r^i(\deg a-i)}$$
for $a\in A_{<d}$ and $0\leq i\leq \deg a$ (see \cite[Propositions 12.11 and 12.13]{Rosen2002}). Thus,$$\left|[a,i]\lambda_\nfk^{r^i-1}\right|_\infty=r^{r^i(\deg a-i)}r^{(r^i-1)(-d+1+\frac{1}{r-1})}=r^{r^i(\deg a-d-i+1+\frac{1}{r-1})+d-1-\frac{1}{r-1}}$$attains its unique maximum at $i=0$. Therefore, we obtain$$|[a]_{\lambda_{\nfk}}-a|_\infty=\left|\sum_{i=1}^{\deg a}[a,i]\lambda_\nfk^{r^i-1}\right|_\infty\leq \max_{1\leq i\leq \deg a}\left|[a,i]\lambda_\nfk^{r^i-1}\right|_\infty<|a|_\infty.$$
In particular,
\[
\left|[a]_{\lambda_\nfk}\right|_\infty=\left|([a]_{\lambda_\nfk}-a)+a\right|_\infty=|a|_\infty.
\]
\end{proof}

Then we can prove that as $\nfk\to\infty$, the ``analytic limit'' of $H_{<\deg\nfk}(\ww{s};\lambda_\nfk)$ is exactly the corresponding Thakur's multiple zeta value $\zeta_A(\ww{s})$. Precisely, we have the following theorem:
\begin{thm}\label{thm:limit of u}
    For any index $\ww{s}\in\Ical^{\ext}$, we have
    $$H_{<\deg\nfk}(\ww{s};\lambda_\nfk)\to \zeta_A(\ww{s})$$
    as $\deg\nfk\to\infty$.
\end{thm}
\begin{proof}
    For the empty index, the result is clear. Let $d=\deg\nfk$. We first assume $\ww{s}=(s_1,\ldots,s_m)\in\Ical$ is positive and non-empty.
    Let $\epsilon_a=[a]_{\lambda_\nfk}-a=\sum_{i=1}^{\deg a}[a,i]\lambda_\nfk^{r^i-1}$. By Lemma \ref{lem:[a]n=a}, we have
    $$\left|[a]_{\lambda_\nfk}\right|_\infty=|a|_\infty\quad
    \text{and}\quad
    \left|\epsilon_a\right|_\infty<|a|_\infty.$$ 
    Thus, for $s\in\NN$, we have
    \begin{align*}
        \left|[a]_{\lambda_\nfk}^{-s}-a^{-s}\right|_\infty&=\left|a^{-s}\left(\left(1+\frac{\epsilon_a}{a}\right)^{-s}-1\right)\right|_\infty.
    \end{align*}
    Note that 
    $$\left(1+\frac{\epsilon_a}{a}\right)^{-s}=1+\sum_{k=1}^\infty \binom{-s}{k}\left(\frac{\epsilon_a}{a}\right)^k,$$
    where the right-hand side converges since $|\epsilon_a|_\infty < |a|_\infty$. Therefore, we obtain
    \begin{align*}
        \left|[a]_{\lambda_\nfk}^{-s}-a^{-s}\right|_\infty&= |a|_\infty^{-s}\left|\sum_{k=1}^\infty \binom{-s}{k}\left(\frac{\epsilon_a}{a}\right)^k\right|_\infty\leq |a|_\infty^{-s}\left|\frac{\epsilon_a}{a}\right|_\infty\\
        &\leq |a|_\infty^{-s-1}\max_{1\leq j\leq \deg a}\left|[a,j]\right|_\infty\left|\lambda_{\nfk}\right|^{r^j-1}_\infty.
    \end{align*}
    Here, note that $\left|\binom{-s}{k}\right|_\infty \leq 1$ and $\left|\epsilon_a/a\right|_\infty<1$.
    By Lemma \ref{lem:prod-identity} and the discussion above, for any fixed $a_1,\ldots,a_m\in A_{+}$ with $d>\deg a_1>\cdots>\deg a_m\geq 0$, we have
    \begin{align*}
        \left|\prod_{i=1}^m[a_i]_{\lambda_\nfk}^{-s_i}-\prod_{i=1}^m a_i^{-s_i}\right|_\infty
        &=\left|\sum_{k=1}^m\left(\prod_{i<k}a_i^{-s_i}\right)\left(\prod_{i>k}[a_i]_{\lambda_\nfk}^{-s_i}\right)\left([a_k]_{\lambda_\nfk}^{-s_k}-a_k^{-s_k}\right)\right|_\infty\\
        &\leq \max_{1\leq k\leq m}\prod_{\substack{i=1\\ i\neq k}}^m\left|a_i\right|_\infty^{-s_i}|a_k|_\infty^{-s_k-1}\max_{1\leq j\leq \deg a_k}\left|[a_k,j]\right|_\infty\left|\lambda_{\nfk}\right|^{r^j-1}_\infty\\
        &\leq \max_{1\leq k\leq m} r^{-2\deg a_k}\max_{1\leq j\leq \deg a_k}r^{r^j(\deg a_k-d-j+1+\frac{1}{r-1})+d-1-\frac{1}{r-1}}\\
        &\leq \max_{1\leq j\leq \deg a_1}r^{r^j(\deg a_1-d-j+1+\frac{1}{r-1})+d-1-\frac{1}{r-1}-2\deg a_1}.
    \end{align*}
    To justify the last two steps, we note that in the penultimate inequality, $|a_i|_\infty^{-s_i} \le 1$ for all $i$ and $-s_k-1 \le -2$ as $s_k \in \mathbb{N}$. For the last inequality, we observe that for any fixed $j \ge 1$, the exponent is a monotonically increasing function of $\deg a_k$, so the maximum is attained at $k=1$.
    
    Therefore, it follows that
    \begin{align*}
        \left|H_{<d}(\ww{s};\lambda_\nfk)-S_{<d}(\ww{s})\right|_\infty&=\left|\sum_{\substack{a_1,\ldots,a_m\in A_+\\ d>\deg a_1>\cdots>\deg a_m\geq 0}}\left(\frac{1}{[a_1]_{\lambda_\nfk}^{s_1}\cdots[a_m]_{\lambda_\nfk}^{s_m}}-\frac{1}{a_1^{s_1}\cdots a_m^{s_m}}\right)\right|_\infty\\
        &\leq \max_{\substack{a_1,\ldots,a_m\in A_+\\d>\deg a_1>\cdots>\deg a_m\geq 0}}\left|\prod_{i=1}^m[a_i]_{\lambda_\nfk}^{-s_i}-\prod_{i=1}^m a_i^{-s_i}\right|_\infty\\
        &\leq \max_{\substack{a_1,\ldots,a_m\in A_+\\d>\deg a_1>\cdots>\deg a_m\geq 0}}\max_{1\leq j\leq \deg a_1}r^{r^j(\deg a_1-d-j+1+\frac{1}{r-1})+d-1-\frac{1}{r-1}-2\deg a_1}\\
        &\leq \max_{1\leq j\leq d-1}r^{r^j(-j+\frac{1}{r-1})-d+1-\frac{1}{r-1}}=r^{r(-1+\frac{1}{r-1})-d+1-\frac{1}{r-1}}\to 0 
    \end{align*}
    as $d\to\infty$, which implies that
    $$H_{<d}(\ww{s};\lambda_\nfk)\to \lim_{d\to\infty}S_{<d}(\ww{s})=\zeta_A(\ww{s}).$$

 It remains to prove the result for non-positive indices $\ww{s}=(s_1,\ldots,s_m)\in\Ical^{\ext}$. Let $i \in \{1, \ldots, m\}$ be the smallest integer such that $s_i\leq 0$. By \cite[Theorem 5.1.2]{Thakur2004}, there exists an integer $N\geq 0$ such that $S_d(s_i)=H_d(s_i; u)=0$ for all $d \geq N$. Now, for $d>N$, we obtain
\begin{align*}
    H_{<d}(\ww{s};u)&=\sum_{d>d_1>\cdots>d_m\geq 0}H_{d_1}(s_1;u)\cdots H_{d_m}(s_m;u)\\
    &=\sum_{d_i=0}^{N-1}H_{d_i}(s_i,\ldots,s_m;u)\sum_{d>d_1>\cdots>d_{i-1}>d_i}H_{d_1}(s_1;u)\cdots H_{d_{i-1}}(s_{i-1};u),
\end{align*}
and similarly,
\begin{align*}
    S_{<d}(\ww{s})&=\sum_{d>d_1>\cdots>d_m\geq 0}S_{d_1}(s_1)\cdots S_{d_m}(s_m)\\
    &=\sum_{d_i=0}^{N-1}S_{d_i}(s_i,\ldots,s_m) \sum_{d>d_1>\cdots>d_{i-1}>d_i}S_{d_1}(s_1)\cdots S_{d_{i-1}}(s_{i-1}).
\end{align*}
Since $i$ is the smallest integer such that $s_i\leq 0$, the preceding entries $s_1,\ldots,s_{i-1}$ are all positive. Thus, by applying almost the same argument as above, we deduce that for each fixed $0\leq d_i\leq N-1$,
$$\left|\sum_{d>d_1>\cdots>d_{i-1}>d_i}H_{d_1}(s_1;\lambda_\nfk)\cdots H_{d_{i-1}}(s_{i-1};\lambda_\nfk)-\sum_{d>d_1>\cdots>d_{i-1}>d_i}S_{d_1}(s_1)\cdots S_{d_{i-1}}(s_{i-1})\right|_\infty\to 0$$
as $d=\deg\nfk \to \infty$. In addition, it follows immediately from definitions that 
$$\lim_{d\to\infty}H_{d_i}(s_i,\ldots,s_m;\lambda_\nfk)=S_{d_i}(s_i,\ldots,s_m).$$
Therefore, we obtain 
$$|H_{<d}(\ww{s};\lambda_\nfk)-S_{<d}(\ww{s})|_\infty\to 0$$
as $d\to\infty$, which implies that 
$$H_{<d}(\ww{s};\lambda_\nfk)\to \lim_{d\to\infty}S_{<d}(\ww{s})=\zeta_A(\ww{s}).$$
\end{proof}

Next, we consider the ``algebraic limit'' of $H_{<\deg\nfk}(\ww{s};\lambda_\nfk)$. We begin by recalling the $K$-algebra
$$\Acal_K=\left(\prod_{v}\FF_v\right)\Bigg/\left(\bigoplus_v \FF_v\right),$$
where $v$ ranges over all monic irreducible polynomials in $A$ and $\FF_v=A/(v)$. In addition, we recall the definition of finite multiple zeta values over $K$, which, for any index $\ww{s}=(s_1,\ldots,s_m)\in\Ical$, are given by
$$\zeta_{\Acal_K}(\ww{s})=\left(\sum_{\substack{a_1,\ldots,a_m\in A_+\\\deg v>\deg a_1>\cdots>\deg a_m\geq 0}}\frac{1}{a_1^{s_1}\cdots a_m^{s_m}}\mod{v}\right)_v\in\Acal_K.$$
Then we have the following theorem:
\begin{thm}\label{thm:u to finite MZV}
    For any index $\ww{s}\in\Ical^{\ext}$, we have
    \begin{equation}\label{eq:H_{<n} mod finite}
        \zeta_{\mathcal{A}_K}(\ww{s})=\left(H_{<\deg v}(\ww{s};\lambda_v)\mod{\lambda_v}\right)_v\in\mathcal{A}_K.
    \end{equation}
\end{thm}
\begin{proof}
   First, note that 
    $H_{<d_v}(\ww{s};\lambda_v)\in K(\Lambda_v),$ whose ring of integers is $A[\lambda_v]$ (see \cite[Proposition 12.9]{Rosen2002}). Furthermore, for all $a\in A_{+,<d_v}$, observe that $\gcd(a,v)=1$, and hence
    $$[a]_{\lambda_v}=\frac{\C_a(\lambda_v)}{\lambda_v}$$
    is a unit in $A[\lambda_v]$ (see \cite[Proposition 12.6]{Rosen2002}), which implies that $H_{<d_v}(\ww{s};\lambda_v)\in A[\lambda_v]$. Next, recall that
    $$vA[\lambda_v]=(\lambda_v)^{\Phi(v)}$$
    is totally ramified where $\Phi(v)=\#(A/(v))^\times=[K(\Lambda_v):K]$ (see \cite[Proposition 12.7]{Rosen2002}). Thus, the inertia degree of $(\lambda_v)$ over $v$ is $f(\lambda_v/v)=1$, so $\FF_v\cong A[\lambda_v]/(\lambda_v)$. Therefore, the right-hand side of \eqref{eq:H_{<n} mod finite} makes sense. Now the result follows immediately from the fact that
    $$[a]_{\lambda_v}=\frac{\C_a(\lambda_v)}{\lambda_v}\equiv a\pmod {\lambda_v}$$
    for all $a\in A_+$.
\end{proof}

Consequently, the ``analytic limit'' of $H_{<\deg\nfk}(\ww{s};\lambda_\nfk)$ becomes the corresponding Thakur's multiple zeta value, whereas its ``algebraic limit'' yields the corresponding finite multiple zeta value over $K$. Recalling the finite Euler--Carlitz formula (Theorem \ref{thm:finite Euler Carlitz}), we obtain the following two corollaries: 

\begin{cor}[Euler-Carlitz formula]\label{cor:Euler-Carlitz}
    Let $s\in\NN$ be a $r$-even integer, i.e., $(r-1)\mid s$. Then 
    $$\frac{\zeta_A(s)}{\cpi^{s}}=\frac{\BC_s}{\Gamma_{s+1}}.$$
\end{cor}
\begin{proof}
    By Theorem~\ref{thm:limit of u}, we have
    \begin{equation}\label{eq:Euler-Carlitz for finite}
        \frac{H_{<\deg\nfk}(s;\lambda_\nfk)}{(\nfk\lambda_\nfk)^s}=\frac{\dBC_s(\nfk)}{\Gamma_{s+1}}.
    \end{equation}
    Observe that
    \[
        \nfk\lambda_\nfk
        =
        \nfk \exp_\bfC(\cpi/\nfk)
        =
        \cpi
        +
        \sum_{i=1}^{\infty}\frac{\cpi^{r^i}}{D_i\,\nfk^{\,r^i-1}}
        \to \cpi,
    \]
     and $\dBC_s(\nfk)\to\BC_s$ as $d=\deg\nfk\to\infty$. By Theorem~\ref{thm:limit of u}, the identity \eqref{eq:Euler-Carlitz for finite} converges, as $d\to\infty$, to the Euler--Carlitz formula.
\end{proof}

\begin{cor}\label{cor:vanishing of fmzv qeven}
    Let $s\in\NN$ be a $r$-even integer, i.e., $(r-1)\mid s$. Then 
    $$\zeta_{\Acal_K}(s)=(0)_v\in\Acal_K.$$
\end{cor}
\begin{proof}
    The result follows immediately from Theorems \ref{thm:finite Euler Carlitz} and \ref{thm:u to finite MZV}.
\end{proof}
\begin{rmk}
    We mention that Corollary \ref{cor:vanishing of fmzv qeven} was previously established in \cite{Shi2018} via a different approach, where the result was derived from an explicit computation of \textit{sign-free truncated multiple zeta values}.
\end{rmk}
Recall that in the classical theory, finite multiple harmonic $q$-series $Z_{<n}(\ww{s};\zeta_n)$ evaluated at roots of unity relate to symmetric multiple zeta values in the analytic limit and to finite multiple zeta values in the algebraic limit. This connection is viewed as evidence for the Kaneko--Zagier conjecture. 

Hence, it is natural to consider a function field analog of the Kaneko--Zagier conjecture.  
We first recall the notions of fixed and binary relations introduced in \cite{Todd2018}. 

\begin{df}
    For $\bullet\in\{\Acal_K,A\}$, a $K$-linear relation 
    \[ \sum_{i}c_i\zeta_{\bullet}(\ww{s}_i)=0,\quad c_i\in K, \ww{s}_i\in\Ical \]
    is called \textit{binary} if there exist $a_i,b_i\in K$ such that $a_i+b_i=c_i$ for each $i$ and 
    \[ \sum_{i} (a_i S_d(\ww{s}_i) + b_i S_{d+1}(\ww{s}_i)) = 0 \]
    for all $d\in\ZZ$. In particular, it is called \textit{fixed} if $b_i=0$ for all $i$.
\end{df}

It is obvious that $\zeta_{\Acal_K}(\ww{s})$ satisfy all fixed relations. Motivated by a similar spirit of  \cite{BTT18} and Shi's computation \cite{Shi2018}, we consider the following question:

\begin{question}[cf. {\cite[Conjecture 4.6.5]{Shi2018}}]\label{conj:Function_field_KZ}
    Does there exist a well-defined surjective $K$-algebra homomorphism
    \[
    \varphi : \Zcal_{\mathcal{A}_K} \longrightarrow \Zcal_\infty / \zeta_A(r-1)\Zcal_\infty  
    \]
    satisfying
    \[
    \varphi\left(\zeta_{\mathcal{A}_K}(\ww{s})\right) \equiv \zeta_A(\ww{s}) \pmod{\zeta_A(r-1)}
    \]
    for any index $\ww{s} \in \Ical$? If so, is  $\ker\varphi$ the ideal generated by all binary relations among $\zeta_{\Acal_K}(\ww{s})$ that are not fixed?
\end{question}

\begin{rmk}
In fact, one may further ask whether $\ker\varphi$ is generated by Thakur's fundamental relation  \cite[Theorem 5]{Thakur2009}. That is, whether the equality
  \[
    \ker\varphi = \left\langle \zeta_{\mathcal{A}_K}(r) + (\theta^r - \theta)\zeta_{\mathcal{A}_K}(1, r-1) \right\rangle
    \]
holds.
\end{rmk}

\section{\texorpdfstring{$u$}{u}-Multiple Zeta Values}\label{section:4}
\subsection{\texorpdfstring{$u$}{u}-Multiple Zeta Values on Drinfeld Upper-Half Plane}\label{section:4.1}
We begin by recalling the \textit{Drinfeld upper-half plane}
$$\Omega := \PP^1(\CC_\infty) \setminus \PP^1(K_\infty) = \CC_\infty \setminus K_\infty,$$
which admits a natural rigid analytic structure (see \cite[Proposition 6.1]{Drinfeld1974} and \cite{Goss1980b}). For $z\in\CC_\infty$, we let  $$|z|_i=\inf\{|z-x|_\infty:x\in K_\infty\}$$
be the \textit{imaginary part} of $z$. Then 
\[
\Omega = \bigcup_{n \in \NN} \Omega_n
\quad
\text{where}
\quad
\Omega_n := \{z \in \CC_\infty \mid  r^{-n} \leq |z|_i \leq |z|_\infty \leq r^n\}
\]
is an admissible cover of $\Omega$. Furthermore, a function $f:\Omega\to\CC_\infty$ is called \textit{rigid analytic} if its restriction to each $\Omega_n$ is a uniform limit of rational functions without poles on $\Omega_n$ (see \cite{Goss1980a} and \cite{FresnelPut2004}). We denote by $\Ocal(\Omega)$ the algebra of rigid analytic functions on $\Omega$.

For any index $\ww{s}\in\Ical$, we can view the $u$-multiple zeta values as functions on $\Omega$ by defining
\[
\zeta_u(\ww{s};z):=\zeta_{\exp_\C(\cpi z)}(\ww{s}),\quad z\in\Omega.
\]
Then $\zeta_u(\ww{s};z)$ is well-defined and is a rigid analytic function on $\Omega$:
\begin{lem}
    For any index $\ww{s}\in\Ical^{\ext}$, $\zeta_u(\ww{s};z)$ is a rigid analytic function on $\Omega$.
\end{lem}

\begin{proof}
    For the empty index, the result is clear. First, suppose that $\ww{s}=(s_1,\ldots,s_m)\in\Ical$ is a non-empty positive index. Then 
    \begin{align*}
        \zeta_u(\ww{s};z)=\zeta_{\exp_\C(\cpi z)}(\ww{s})&=\lim_{d\to\infty}\sum_{\substack{a_1,\ldots,a_m\in A_+\\d>\deg a_1>\cdots>\deg a_m\geq 0}}\frac{\exp_\C(\cpi z)^{s_1+\cdots+s_m}}{\C_{a_1}(\exp_\C(\cpi z))^{s_1}\cdots \C_{a_m}(\exp_\C(\cpi z))^{s_m}}\\
        &=\lim_{d\to\infty}\sum_{\substack{a_1,\ldots,a_m\in A_+\\d>\deg a_1>\cdots>\deg a_m\geq 0}}\frac{\exp_\C(\cpi z)^{s_1+\cdots+s_m}}{\exp_{\C}(\cpi a_1z)^{s_1}\cdots\exp_{\C}(\cpi a_mz)^{s_m}}.
    \end{align*}
    Now, let $n\in\NN$ and $z\in\Omega_n$. Note that for $a\in A_+$ with $\deg a>n+1$, we have $|az|_{\infty}\geq r>1$. By \cite[Lemma 5.5]{Gekeler1988}, 
    \[
    \log_r|\exp_\C(\cpi az)|_\infty\geq |az|_i=|a|_\infty|z|_i\to\infty
    \]
    as $\deg a\to\infty$. Thus, $\zeta_u(\ww{s};z)$ is a uniform limit of rational functions without poles on $\Omega_n$, and hence it is a rigid analytic function on $\Omega$.

    Next, let $\ww{s}=(s_1,\ldots,s_m)\in\Ical^{\ext}$ be a non-positive index, and let $i \in \{1, \ldots, m\}$ be the smallest integer such that $s_i\leq 0$. By \cite[Theorem 5.1.2]{Thakur2004}, there exists an integer $N\geq 0$ such that $H_d(s_i; \exp_\C(\cpi z))=0$ for all $d \geq N$. Thus, for $d>N$, we have
    \begin{align*}
        \zeta_u(\ww{s};z)&=\lim_{d\to\infty}H_{<d}(\ww{s};\exp_\C(\cpi z))\\
        &=\lim_{d\to\infty}\sum_{d>d_1>\cdots>d_m\geq 0}H_{d_1}(s_1;\exp_\C(\cpi z))\cdots H_{d_m}(s_m;\exp_\C(\cpi z))\\
        &=\lim_{d\to\infty}\sum_{d_i=0}^{N-1}H_{d_i}(s_i,\ldots,s_m;\exp_\C(\cpi z))\\
        &\qquad\times\sum_{d>d_1>\cdots>d_{i-1}>d_i}H_{d_1}(s_1;\exp_\C(\cpi z))\cdots H_{d_{i-1}}(s_{i-1};\exp_\C(\cpi z))\\
        &=\sum_{d_i=0}^{N-1}c_{d_i}(z)\lim_{d\to\infty}\sum_{\substack{a_1,\ldots,a_{i-1}\in A_+\\d>\deg a_1>\cdots>\deg a_{i-1}>d_i}}\frac{\exp_\C(\cpi z)^{s_1+\cdots+s_{i-1}}}{\exp_{\C}(\cpi a_1z)^{s_1}\cdots\exp_{\C}(\cpi a_{i-1}z)^{s_{i-1}}}
    \end{align*}
    where $c_{d_i}(z)=H_{d_i}(s_i,\ldots,s_m;\exp_\C(\cpi z))$ is clearly rigid analytic.
    It follows by a similar estimate that $\zeta_u(\ww{s};z)$ is a rigid analytic function on $\Omega$.
\end{proof}
By Corollary \ref{cor:zeta u shuffle}, we have the following immediate corollary. 
\begin{cor}
  The unique $\FF_p$-linear map
$$\zeta_u(\bullet;z):\Rcal\to \Ocal(\Omega)$$  
sending $x_{\ww{s}}$ to $\zeta_u(\ww{s};z)$, $\ww{s}\in\Ical$, is an $\FF_p$-algebra homomorphism.
\end{cor}

Next, recall that for a rigid analytic function $f: \Omega \to \CC_\infty$ that is $A$-periodic (i.e., $f(z+a) = f(z)$ for all $a \in A$), $f(z)$ admits a unique $t$-expansion
$$ f(z) = \sum_{N \in \ZZ} a_N t^N,\quad a_N\in\CC_\infty$$
which converges for $|z|_i$ sufficiently large (see \cite{Goss1980a}). Here, we let the local parameter at infinity be
$$ t = t(z) := \frac{1}{\exp_\C(\cpi z)}. $$

One checks easily that $\zeta_u(\ww{s};z)$ is $A$-periodic. We now compute the $t$-expansion of $\zeta_u(\ww{s};z)$ at positive index.

\begin{thm}\label{thm:t_expansion}
    For any non-empty index $\ww{s}=(s_1,\ldots,s_m)\in\Ical$, the rigid analytic function $\zeta_u(\ww{s}; z)$ admits a $t$-expansion with coefficients in $A$:
    \[ \zeta_u(\ww{s}; z) = \sum_{N=0}^\infty c_N t^N \in A[[t]] \]
    where $c_0 = 1$ if $m=1$, and $c_0 = 0$ if $m \geq 2$. 
\end{thm}

\begin{proof}
     Evaluating at $u=\exp_\C(\cpi z)=1/t$, we have
    \[ [a]_{1/t} = t \C_a(1/t) = t \sum_{i=0}^d [a,i] t^{-r^i} = t^{1-r^d} \left( 1 + \sum_{i=0}^{d-1} [a,i] t^{r^d - r^i} \right). \]
    Note that $1 + \sum_{i=0}^{d-1} [a,i] t^{r^d - r^i}$ is invertible in $A[[t]]$. Therefore, for any $s\in\NN$, 
    \[ \frac{1}{[a]_{1/t}^s} \in t^{s(r^d-1)} A[[t]]. \]
    Now, we substitute this into the definition of $\zeta_u(\ww{s}; z)$:
    \[ \zeta_u(\ww{s}; z) = \sum_{d_1 > \cdots > d_m \geq 0} \sum_{\substack{a_1, \ldots, a_m \in A_+ \\ \deg a_j = d_j}} \prod_{j=1}^m \frac{1}{[a_j]_{1/t}^{s_j}}. \]
    For any fixed sequence $d_1 > \cdots > d_m \geq 0$, the inner sum lies in $t^M A[[t]]$, where $M = \sum_{j=1}^m s_j (r^{d_j}-1)$. When $m=1$, the only term that gives $M=0$ is $d_1=0$, which contributes exactly $1$ to the sum, and hence $c_0=1$. When $m \ge 2$, we necessarily have $d_1 \ge 1$, which implies $M \ge s_1(r-1) \ge 1$. Hence, there is no constant term. Finally, since $M \to \infty$ as $d_1 \to \infty$, each coefficient $c_N$ involves only finitely many terms and therefore lies in $A$.
\end{proof}
\subsection{\texorpdfstring{$u$}{u}-Multiple Zeta Values as Formal Power Series}\label{section:4.2}
We now discuss the expansion of $u$-multiple zeta values as formal power series in $u$. Recall that the $u$-multiple zeta values are defined by 
\begin{equation*}
    \zeta_u(\ww{s})=\lim_{d\to\infty}H_{<d}(\ww{s};u)=\lim_{d\to\infty}\sum_{d>d_1>\cdots>d_m\geq 0}H_{d_1}(s_1;u)\cdots H_{d_m}(s_m;u)
\end{equation*}
for any index $\ww{s}=(s_1,\ldots,s_m)\in\Ical^{\ext}$.
We will show that this defines a formal power series with respect to the product topology on $\CC_\infty\mbb{u}$. 

First, for each integer $i \in \NN$, we define the polynomial $P_i(X) \in K[X]$ by 
$$P_i(X) := \sum_{j=0}^i \frac{X^{r^j-1}}{D_j L_{i-j}^{r^j}}$$
For any $n \in \NN$ and $s\in\ZZ$, we define the polynomial $W_n^{(s)}(X) \in K[X]$ explicitly by 
\begin{equation}\label{eq:W_ns}
W_n^{(s)}(X) := \sum_{k=1}^n \binom{-s}{k} \sum_{\substack{i_1, \dots, i_k \ge 1 \\ \sum_{\ell=1}^k \frac{r^{i_\ell}-1}{r-1} = n}} \prod_{\ell=1}^k P_{i_\ell}(X).    
\end{equation}
We put $W_0^{(s)}(X) = 1$ by convention. Moreover, we write
\begin{equation}\label{eq:W_n coefficient}
    W_{n}^{(s)}(X)=\sum_{i=0}^{n(r-1)}c_{n,i}^{(s)}X^{i},\quad c_{n,i}^{(s)}\in K.
\end{equation}
Here, note that since $\deg P_j(X) = r^j - 1$, the degree of each product in \eqref{eq:W_ns} is given by
$$\deg \left( \prod_{\ell=1}^k P_{i_\ell}(X) \right) = \sum_{\ell=1}^k (r^{i_\ell} - 1) = (r-1) \sum_{\ell=1}^k \frac{r^{i_\ell}-1}{r-1} = n(r-1),$$
where the indices $i_1, \dots, i_k$ satisfy the condition in \eqref{eq:W_ns}. Consequently, we have $$\deg W_n^{(s)}(X) \leq n(r-1).$$

We first consider the lemmas below:

\begin{lem}\label{lem:convergence_W}
Let $n\in \NN$ and $s\in\ZZ$. Then
$ H^W_{d,n}(s) \to 0$ as $d\to\infty$
where $$H^W_{d,n}(s) := \sum_{a \in A_{+,d}} \frac{W_n^{(s)}(a)}{a^s}.$$
\end{lem}
\begin{proof}
By \eqref{eq:W_n coefficient}, we have $$\frac{W_n^{(s)}(X)}{X^s} = \sum_{i=0}^{n(r-1)} c_{n,i}^{(s)} X^{i-s},$$ 
and hence 
\begin{align*}
H^W_{d,n}(s) = \sum_{i=0}^{n(r-1)} c_{n,i}^{(s)} \left( \sum_{a \in A_{+,d}} a^{i-s} \right) =\sum_{i=0}^{n(r-1)} c_{n,i}^{(s)}  S_d(i-s).
\end{align*}
The result follows immediately from the fact that for any $s\in\ZZ$, $S_d(s)\to 0$ as $d\to\infty$. 
\end{proof}

\begin{lem}\label{lem:Carlitz coefficient}Let $a\in A$. Then we have
$$\C_a(u)=\sum_{i=0}^{\deg a}[a,i]u^{r^i}=\sum_{i=0}^\infty \left(\sum_{j=0}^i \frac{a^{r^j}}{D_j L_{i-j}^{r^j}}\right)u^{r^i}$$
where $L_i=\prod_{j=1}^i(\theta-\theta^{r^j})$.
\end{lem}
\begin{proof}
The proof is identical to that in \cite[\S2.5]{Thakur2004}, and is included here for completeness. We first recall the formal Carlitz exponential and logarithm (see \cite{Goss1996})
    $$\exp_\C(u) = \sum_{i=0}^{\infty} \frac{u^{r^i}}{D_i}, \quad \log_\C(u) = \sum_{j=0}^{\infty} \frac{u^{r^j}}{L_j}.$$
    Then we have the following identity in $\CC_\infty\mbb{u}$:
    $$\C_a(u)=\exp_\C(a\log_\C(u))= \sum_{i=0}^{\infty} \frac{1}{D_i} \left( a \sum_{j=0}^{\infty} \frac{u^{r^j}}{L_j} \right)^{r^i}=  \sum_{k=0}^{\infty} \left( \sum_{i=0}^k \frac{a^{r^i}}{D_i L_{k-i}^{r^i}} \right) u^{r^k}.$$
\end{proof}

With these lemmas in hand, we obtain the following expansion of $u$-multiple zeta values in $\CC_\infty\mbb{u}$:
\begin{prop}\label{prop:expansion of umzv}
Let $\ww{s}=(s_1,\ldots,s_m)\in\Ical^{\ext}$. Then $$\zeta_u(\ww{s}) = \sum_{N=0}^{\infty} \gamma_N(\ww{s}) u^{N(r-1)}\in\CC_\infty\mbb{u}$$where the explicit coefficient $\gamma_N(\ww{s})$ is given by \begin{equation}\label{eq:gammaN}
    \gamma_N(\ww{s}) =  \sum_{\substack{n_1, \dots, n_m \ge 0 \\ n_1 + \dots + n_m = N}} \sum_{d_1 > \dots > d_m \ge 0} \prod_{j=1}^m H^W_{d_j,n_j}(s_j) .
\end{equation}
\end{prop}
\begin{proof}Let $a\in A_+$ and $s\in\ZZ$. By Lemma \ref{lem:Carlitz coefficient}, we have
$$[a]_u = \frac{1}{u} \left( au + \sum_{i=1}^{\infty} [a,i] u^{r^i} \right) = a \left( 1 + \sum_{i=1}^{\infty} P_i(a) u^{r^i-1} \right).$$
We first expand the formal series for $[a]_u^{-s}$ using the generalized binomial theorem:
$$[a]_u^{-s} = \frac{1}{a^s} \sum_{k=0}^{\infty} \binom{-s}{k} \left( \sum_{i=1}^{\infty} P_i(a) u^{r^i-1} \right)^k = \frac{1}{a^s} \sum_{k=0}^{\infty} \binom{-s}{k} \left( \sum_{i=1}^{\infty} P_i(a) (u^{r-1})^{\frac{r^i-1}{r-1}} \right)^k.$$
It follows immediately from the definition in \eqref{eq:W_ns} that 
\begin{equation}\label{eq:local_exp}[a]_u^{-s} = \sum_{n=0}^{\infty} \frac{W_n^{(s)}(a)}{a^s} u^{n(r-1)} \in K\mbb{u}.\end{equation}
Now, for any fixed degree $d \ge 0$, we have
$$H_d(s;u) = \sum_{a \in A_{+,d}} \left( \sum_{n=0}^{\infty} \frac{W_n^{(s)}(a)}{a^s} u^{n(r-1)} \right)= \sum_{n=0}^{\infty} H^W_{d,n}(s)  u^{n(r-1)}.$$
For any $d_1 > \dots > d_m \ge 0$, we have
\begin{align*}\prod_{j=1}^m H_{d_j}(s_j;u) &= \prod_{j=1}^m \left( \sum_{n_j=0}^{\infty} H^W_{d_j,n_j}(s_j) u^{n_j(r-1)} \right) \\&= \sum_{N=0}^{\infty} \left( \sum_{\substack{n_1, \dots, n_m \ge 0 \\ n_1+\dots+n_m=N}} \prod_{j=1}^m H^W_{d_j,n_j}(s_j) \right) u^{N(r-1)}.
\end{align*}
It follows that 
\begin{align*}
    \zeta_u(\ww{s}) &=\lim_{d\to\infty} H_{<d}(s_1,\ldots,s_m;u)\\ &=\lim_{d\to\infty}\sum_{d>d_1 > \dots > d_m \ge 0} \left( \sum_{N=0}^{\infty} \left( \sum_{\substack{n_1, \dots, n_m \ge 0 \\ n_1+\dots+n_m=N}} \prod_{j=1}^m H^W_{d_j,n_j}(s_j) \right) u^{N(r-1)} \right)\\
    &=\lim_{d\to\infty}\sum_{N=0}^{\infty} \left( \sum_{\substack{n_1, \dots, n_m \ge 0 \\ n_1 + \dots + n_m = N}} \sum_{d>d_1 > \dots > d_m \ge 0} \prod_{j=1}^m H^W_{d_j,n_j}(s_j) \right) u^{N(r-1)}
\end{align*}

By Lemma \ref{lem:convergence_W}, we see that the series converges with respect to the product topology and this exactly yields $\zeta_u(\ww{s}) = \sum_{N=0}^\infty \gamma_N(\ww{s}) u^{N(r-1)}$, completing the proof.
\end{proof}

We mention that the constant in the formal expansion of $\zeta_u(\ww{s})$ is precisely $\zeta_A(\ww{s})$.

As observed in the proof of Lemma \ref{lem:convergence_W}, the coefficients $\gamma_N(\ww{s})$ can be, in fact, expressed as finite $K$-linear combinations of Thakur's multiple zeta values.

\begin{prop}\label{prop:gamma mzv}
For any $N \ge 0$ and non-empty $\ww{s} = (s_1, \dots, s_m) \in \Ical^{\ext}$, the coefficient $\gamma_N(\ww{s})$ can be written as a $K$-linear combination of Thakur's multiple zeta values
$$\gamma_N(\ww{s}) = \sum_{\substack{n_1, \dots, n_m \ge 0 \\ n_1 + \dots + n_m = N}} \sum_{i_1=0}^{n_1(r-1)} \cdots \sum_{i_m=0}^{n_m(r-1)} \left( \prod_{j=1}^m c_{n_j,i_j}^{(s_j)} \right) \zeta_A(s_1-i_1, \dots, s_m-i_m).$$
\end{prop}

\begin{proof}
By \eqref{eq:W_n coefficient}, we have
$$H^W_{d,n}(s) = \sum_{a \in A_{+,d}} \frac{W_n^{(s)}(a)}{a^s} = \sum_{i=0}^{n(r-1)} c_{n,i}^{(s)} \left( \sum_{a \in A_{+,d}} \frac{1}{a^{s-i}} \right) = \sum_{i=0}^{n(r-1)} c_{n,i}^{(s)} S_d(s-i).$$
Substituting this expansion into the formula for $\gamma_N(\ww{s})$, we obtain
\begin{align*}
\gamma_N(\ww{s}) &= \sum_{\substack{n_1, \dots, n_m \ge 0 \\ n_1 + \dots + n_m = N}} \sum_{d_1 > \dots > d_m \ge 0} \prod_{j=1}^m \left( \sum_{i_j=0}^{n_j(r-1)} c_{n_j,i_j}^{(s_j)} S_{d_j}(s_j-i_j) \right) \\
&= \sum_{\substack{n_1, \dots, n_m \ge 0 \\ n_1 + \dots + n_m = N}} \sum_{i_1=0}^{n_1(r-1)} \cdots \sum_{i_m=0}^{n_m(r-1)} \left( \prod_{j=1}^m c_{n_j,i_j}^{(s_j)} \right) \sum_{d_1 > \dots > d_m \ge 0} \prod_{j=1}^m S_{d_j}(s_j-i_j)\\
&=\sum_{\substack{n_1, \dots, n_m \ge 0 \\ n_1 + \dots + n_m = N}} \sum_{i_1=0}^{n_1(r-1)} \cdots \sum_{i_m=0}^{n_m(r-1)} \left( \prod_{j=1}^m c_{n_j,i_j}^{(s_j)} \right) \zeta_A(s_1-i_1, \dots, s_m-i_m).
\end{align*}
\end{proof}

We compute the coefficient $\gamma_N(\ww{s})$ for $N=1$ and $\dep(\ww{s}) \leq 2$ in the following two examples.
\begin{ex}\label{ex:expansion gamma}
Let $s\in\ZZ$. We have
$$\zeta_u(s) = \sum_{N=0}^{\infty} \gamma_N(s) u^{N(r-1)},$$
where 
$$\gamma_N(s) = \sum_{i=0}^{N(r-1)} c_{N,i}^{(s)} \zeta_A(s-i).$$
We compute the first two terms explicitly:
\begin{enumerate}
    \item For $N=0$, we have $W_0^{(s)}(X) = 1$, which means $c_{0,0}^{(s)} = 1$. Thus, the constant term is exactly Thakur's zeta value:
    $$\gamma_0(s) = \zeta_A(s).$$
    
    \item For $N=1$, the condition $\frac{r^{i_1}-1}{r-1} = 1$ forces $k=1$ and $i_1=1$, so we obtain
    $$W_1^{(s)}(X) = \binom{-s}{1} P_1(X) = -s \sum_{j=0}^1 \frac{X^{r^j-1}}{D_j L_{1-j}^{r^j}} = -s \left( \frac{1}{L_1} + \frac{X^{r-1}}{D_1} \right).$$
    This gives the coefficients $c_{1,0}^{(s)} = \frac{-s}{L_1}$ and $c_{1,r-1}^{(s)} = \frac{-s}{D_1}$. Consequently, the coefficient of $u^{r-1}$ is given by
    $$\gamma_1(s) = -s \left( \frac{1}{L_1} \zeta_A(s) + \frac{1}{D_1} \zeta_A(s-r+1) \right).$$
\end{enumerate}
\end{ex}

\begin{ex}\label{ex:expansion gamma depth 2}
Let $\ww{s} = (s_1, s_2) \in \Ical^{\ext}$. We have
$$\zeta_u(s_1, s_2) = \sum_{N=0}^{\infty} \gamma_N(s_1, s_2) u^{N(r-1)}.$$
By Proposition \ref{prop:gamma mzv}, the coefficient $\gamma_N(s_1, s_2)$ is given by 
$$\gamma_N(s_1, s_2) = \sum_{n_1+n_2=N} \sum_{i_1=0}^{n_1(r-1)} \sum_{i_2=0}^{n_2(r-1)} c_{n_1,i_1}^{(s_1)} c_{n_2,i_2}^{(s_2)} \zeta_A(s_1-i_1, s_2-i_2).$$
We compute the first two terms:
\begin{enumerate}
    \item For $N=0$, since $c_{0,0}^{(s)} = 1$ for any $s$, we obtain 
    $$\gamma_0(s_1, s_2) = \zeta_A(s_1, s_2).$$

    \item For $N=1$, there are two partitions for $n_1 + n_2 = 1$, namely $(1, 0)$ and $(0, 1)$. 
    \begin{enumerate}
        \item When $(n_1, n_2) = (1, 0)$, we have $i_2 = 0$ and $c_{0,0}^{(s_2)}=1$. The index $i_1$ can be $0$ or $r-1$. Using the coefficients $c_{1,0}^{(s_1)}$ and $c_{1,r-1}^{(s_1)}$ computed in Example \ref{ex:expansion gamma}, this part contributes:
        $$-\frac{s_1}{L_1} \zeta_A(s_1, s_2) - \frac{s_1}{D_1} \zeta_A(s_1-r+1, s_2).$$
        \item Similarly, when $(n_1, n_2) = (0, 1)$, we have $i_1 = 0$ and $c_{0,0}^{(s_1)}=1$. The index $i_2$ can be $0$ or $r-1$, which contributes:
        $$-\frac{s_2}{L_1} \zeta_A(s_1, s_2) - \frac{s_2}{D_1} \zeta_A(s_1, s_2-r+1).$$
    \end{enumerate}
    Summing these two contributions, the coefficient of $u^{r-1}$ is given by
    $$\gamma_1(s_1, s_2) = -\frac{s_1+s_2}{L_1} \zeta_A(s_1, s_2) - \frac{s_1}{D_1} \zeta_A(s_1-r+1, s_2) - \frac{s_2}{D_1} \zeta_A(s_1, s_2-r+1).$$
\end{enumerate}
\end{ex}

Next, we consider the $r$-shuffle relations. Recall in the proof of Proposition \ref{prop:expansion of umzv}, we have
$$H_{<d}(\ww{s};u) = \sum_{N=0}^{\infty} \left( \sum_{\substack{n_1, \dots, n_m \ge 0 \\ n_1 + \dots + n_m = N}} \sum_{d > d_1 > \dots > d_m \ge 0} \prod_{j=1}^m H^W_{d_j,n_j}(s_j) \right) u^{N(r-1)}$$
for any positive index $\ww{s}\in\Ical$. Furthermore, by Lemma \ref{lem:convergence_W}, this sequence converges to $\zeta_u(\ww{s})$ coefficient-wise as $d \to \infty$. Therefore, we can apply the theory in \S \ref{section:3.2} and obtain the following corollaries:
\begin{cor}\label{cor:Hu shuffle formal}
    For any $d\in\NN$, the map $$\widehat{H_{<d}}(\bullet;u):\Rcal\to \CC_\infty\mbb{u}$$ is an $\FF_p$-algebra homomorphism. That is,
    $$\widehat{H_{<d}}(x_{\ww{r}}\ast x_{\ww{s}};u) =\widehat{H_{<d}}(x_{\ww{r}};u)\widehat{H_{<d}}(x_{\ww{s}};u)=H_{<d}(\ww{r};u)H_{<d}(\ww{s};u)$$
    for any positive indices $\ww{r},\ww{s}\in\Ical$.
\end{cor}

\begin{cor}\label{cor:zeta u shuffle formal}
    The map $$\widehat{\zeta_u}:\Rcal\to\CC_\infty\mbb{u}$$ is an $\FF_p$-algebra homomorphism. That is,
    $$\widehat{\zeta_u}(x_{\ww{r}}\ast x_{\ww{s}};u) =\widehat{\zeta_u}(x_{\ww{r}};u)\widehat{\zeta_u}(x_{\ww{s}};u)=\zeta_u(\ww{r};u)\zeta_u(\ww{s};u)$$
    for any positive indices $\ww{r},\ww{s}\in\Ical$.
\end{cor}

By exploiting the $r$-shuffle relations of the formal power series $\zeta_u(\ww{s})$, we can extract explicit relations among the coefficients $\gamma_N(\ww{s})$. To formulate this conveniently, we let $\widehat{\gamma_N}: \Rcal \to \CC_\infty$ be the $\FF_p$-linear map defined by $\widehat{\gamma_N}(x_{\ww{s}}) = \gamma_N(\ww{s})$ for any positive index $\ww{s} \in \Ical$.

\begin{thm}\label{thm:gamma_shuffle}
    For any positive indices $\ww{r}, \ww{s} \in \Ical$ and any integer $N \ge 0$, we have 
    $$\widehat{\gamma_N}(x_{\ww{r}} \ast x_{\ww{s}}) = \sum_{k=0}^N \widehat{\gamma_k}(x_{\ww{r}}) \widehat{\gamma_{N-k}}(x_{\ww{s}}).$$
\end{thm}

\begin{proof}
    By Corollary \ref{cor:zeta u shuffle formal}, we have $\widehat{\zeta_u}(x_{\ww{r}} \ast x_{\ww{s}}) = \zeta_u(\ww{r}) \zeta_u(\ww{s})$. We expand both sides as formal power series in $u$. For the left-hand side, we have
    $$\widehat{\zeta_u}(x_{\ww{r}} \ast x_{\ww{s}}; u) = \sum_{N=0}^{\infty} \widehat{\gamma_N}(x_{\ww{r}} \ast x_{\ww{s}}) u^{N(r-1)}.$$
    For the right-hand side, we have
    \begin{align*}
        \zeta_u(\ww{r}; u) \zeta_u(\ww{s}; u) &= \left( \sum_{k=0}^{\infty} \gamma_k(\ww{r}) u^{k(r-1)} \right) \left( \sum_{\ell=0}^{\infty} \gamma_\ell(\ww{s}) u^{\ell(r-1)} \right) \\
        &= \sum_{N=0}^{\infty} \left( \sum_{k=0}^N \gamma_k(\ww{r}) \gamma_{N-k}(\ww{s}) \right) u^{N(r-1)}.
    \end{align*}
    Comparing the coefficients of $u^{N(r-1)}$ on both sides yields the desired equality.
\end{proof}

Theorem \ref{thm:gamma_shuffle}, together with Proposition \ref{prop:gamma mzv}, provides a specific method to produce relations for Thakur's multiple zeta values, including values at non-positive indices. We first illustrate this method in the cases \(N=1\) and \(N=2\), where the resulting identities take particularly simple forms.

\begin{ex}\label{ex:gamma shuffle detailed}
To see how multiple zeta values at non-positive indices arise from Theorem \ref{thm:gamma_shuffle}, we consider the $r$-shuffle product of two depth-one words $x_{r_1}$ and $x_{s_1}$ for some $r_1,s_1\in\NN$. By Theorem \ref{thm:gamma_shuffle}, we extract the relation for $N=1$, which gives the equation
\begin{equation}\label{eq:gamma1 base eq}
    \widehat{\gamma_1}(x_{r_1} \ast x_{s_1}) = \gamma_1(r_1)\gamma_0(s_1) + \gamma_0(r_1)\gamma_1(s_1).
\end{equation}
We now express both sides of \eqref{eq:gamma1 base eq} completely in terms of multiple zeta values using Examples \ref{ex:expansion gamma} and \ref{ex:expansion gamma depth 2}.

First, recall from Example \ref{ex:expansion gamma} that $\gamma_0(s) = \zeta_A(s)$ and $$\gamma_1(s) = -s \left( \frac{1}{L_1} \zeta_A(s) + \frac{1}{D_1} \zeta_A(s-r+1) \right).$$ 
Substituting these into the right-hand side of \eqref{eq:gamma1 base eq} yields
\begin{align}
    \notag&\left[ -r_1 \left( \frac{1}{L_1} \zeta_A(r_1) + \frac{1}{D_1} \zeta_A(r_1-r+1) \right) \right] \zeta_A(s_1) + \zeta_A(r_1) \left[ -s_1 \left( \frac{1}{L_1} \zeta_A(s_1) + \frac{1}{D_1} \zeta_A(s_1-r+1) \right) \right] \\\label{eq:RHS}
    &= -\frac{r_1+s_1}{L_1} \zeta_A(r_1)\zeta_A(s_1) - \frac{1}{D_1} \Big( r_1\zeta_A(r_1-r+1)\zeta_A(s_1) + s_1\zeta_A(r_1)\zeta_A(s_1-r+1) \Big).
\end{align}

Next, note that
\begin{equation}\label{eq:ast expansion}
    \widehat{\gamma_1}(x_{r_1} \ast x_{s_1}) = \gamma_1(r_1, s_1) + \gamma_1(s_1, r_1) + \gamma_1(r_1+s_1) + \sum_{\substack{i,j\geq 1\\i+j=r_1+s_1}}\Delta_{r_1,s_1}^{i,j}\gamma_1(i, j).
\end{equation}
We apply the explicit formulas from Example \ref{ex:expansion gamma} and \ref{ex:expansion gamma depth 2} to each term. Then \eqref{eq:ast expansion} becomes
\begin{align}
    \notag-\frac{r_1+s_1}{L_1} &\zeta_A(r_1)\zeta_A(s_1) \\\notag
    &- \frac{1}{D_1} \Big( r_1\zeta_A(r_1-r+1,s_1) + s_1\zeta_A(r_1,s_1-r+1) \Big) \\\notag
    &- \frac{1}{D_1} \Big( s_1\zeta_A(s_1-r+1,r_1) + r_1\zeta_A(s_1,r_1-r+1) \Big) \\\notag
    &- \frac{r_1+s_1}{D_1} \zeta_A(r_1+s_1-r+1) \\\label{eq:LHS}
    &- \frac{1}{D_1} \sum_{\substack{i,j\geq 1\\i+j=r_1+s_1}} \Delta_{r_1,s_1}^{i,j} \Big( i\zeta_A(i-r+1,j) + j\zeta_A(i,j-r+1) \Big).
\end{align}

Finally, equating \eqref{eq:RHS} and \eqref{eq:LHS}, we obtain the identity:
\begin{align}
    \notag&r_1\zeta_A(r_1-r+1)\zeta_A(s_1) + s_1\zeta_A(r_1)\zeta_A(s_1-r+1) \\
    \notag&= r_1\zeta_A(r_1-r+1,s_1) + s_1\zeta_A(r_1,s_1-r+1) + s_1\zeta_A(s_1-r+1,r_1) + r_1\zeta_A(s_1,r_1-r+1) \\\label{eq:derivation}
    &\quad + (r_1+s_1)\zeta_A(r_1+s_1-r+1) + \sum_{\substack{i,j\geq 1\\i+j=r_1+s_1}}\Delta_{r_1,s_1}^{i,j} \Big( i\zeta_A(i-r+1,j) + j\zeta_A(i,j-r+1) \Big).
\end{align}
If $r_1\leq r-1$ or $s_1 \leq  r-1$, then \eqref{eq:derivation} involves Thakur's multiple zeta values at non-positive arguments.
\end{ex}

\begin{ex}Let $r_1, s_1 \in \mathbb{N}$. By extracting the coefficient of $N=2$ from the shuffle relation $x_{r_1} \ast x_{s_1}$ in Theorem \ref{thm:gamma_shuffle}, one can obtain the identity (cf. \eqref{eq:derivation}):
    \begin{align}
\notag& r_1 s_1 \zeta_A(r_1-r+1)\zeta_A(s_1-r+1) + \binom{-r_1}{2}\zeta_A(r_1-2r+2)\zeta_A(s_1) + \binom{-s_1}{2}\zeta_A(r_1)\zeta_A(s_1-2r+2) \\
\notag&= r_1 s_1 \Big( \zeta_A(r_1-r+1, s_1-r+1) + \zeta_A(s_1-r+1, r_1-r+1) \Big) \\
\notag&\quad + \binom{-r_1}{2} \Big( \zeta_A(r_1-2r+2, s_1) + \zeta_A(s_1, r_1-2r+2) \Big) \\
\notag&\quad + \binom{-s_1}{2} \Big( \zeta_A(r_1, s_1-2r+2) + \zeta_A(s_1-2r+2, r_1) \Big) \\
\notag&\quad + \binom{-(r_1+s_1)}{2} \zeta_A(r_1+s_1-2r+2) \\
\label{eq:N2_final_derivation}&\quad + \sum_{\substack{i,j\geq 1\\i+j=r_1+s_1}} \Delta_{r_1,s_1}^{i,j} \left( i j \zeta_A(i-r+1, j-r+1) + \binom{-i}{2} \zeta_A(i-2r+2, j) + \binom{-j}{2} \zeta_A(i, j-2r+2) \right).
\end{align}
\end{ex}

\begin{rmk}
For $N=1,2$, the above identities may simplify after expressing the
terms at non-positive indices in terms of multiple zeta values at
positive indices (see Lemma
\ref{lem:linearcombination-non-positive-index}). Moreover, the usual
weight need not be preserved under this reduction. In some cases, the
resulting identities become tautological or follow from the usual
$r$-shuffle relations. For larger $N$, more complicated identities may
occur (see Remark \ref{rmk:419}).
\end{rmk}
\subsection{Derivations on \texorpdfstring{$r$}{r}-shuffle Relations}\label{section:4.3}

Inspired by \eqref{eq:derivation} and \eqref{eq:N2_final_derivation}, we explicitly define a new family of operators $\Dcal_N$. We will prove that they form a Hasse-Schmidt derivation over $\widehat{\zeta_A}$ by constructing a new bracket compatible with our abstract framework (see Corollary \ref{cor:Z hat shuffle}).

Let us first recall the general notion of a Hasse-Schmidt derivation (see \cite{Weisfeld1965}):

\begin{df}

    Let $\KK$ be a field, and let $A, B$ be $\KK$-algebras. Given a $\KK$-algebra homomorphism $f: A \to B$, a sequence of $\KK$-linear maps $(D_N)_{N \ge 0}$ from $A$ to $B$ is called a \textit{Hasse-Schmidt derivation over $f$} if

    \begin{enumerate}

        \item $D_0 = f$, and

        \item $D_N(xy) = \sum_{k=0}^N D_k(x)D_{N-k}(y)$ for all $x, y \in A$ and $N \geq 0$.

    \end{enumerate}

\end{df}

Recall that $\Zcal_\infty$ is the $K$-algebra generated by all $\zeta_A(\ww{s})$ for $\ww{s}\in\Ical$.  Then the operators $\Dcal_N$ are defined as follows:
\begin{df}\label{df:D_N}
    For $N\geq 0$, we define $\Dcal_N : \Rcal \to \Zcal_\infty$ to be the unique $\FF_p$-linear map such that
    $\Dcal_N(x_{\varnothing})=\delta_{N,0},$
    and 
    $$\Dcal_N(x_{\ww{s}}) := \sum_{\substack{n_1, \ldots, n_m \ge 0 \\ n_1 + \dots + n_m = N}} \left( \prod_{j=1}^m \binom{-s_j}{n_j} \right) \zeta_A(s_1-n_1(r-1), \dots, s_m-n_m(r-1))$$
    for non-empty positive index $\ww{s}=(s_1,\ldots,s_m)\in\Ical$.
\end{df}

We remark that every Thakur multiple zeta value at a non-positive index is a $K$-linear combination of Thakur multiple zeta values at positive indices, so $\Dcal_N$ are well-defined. This is explained by the following lemma:

\begin{lem}\label{lem:linearcombination-non-positive-index} Let $\ww{s}\in\Ical^{\ext}$ be any non-positive index. Then for all $d\gg 0$,
$$S_{<d}(\ww{s})=\sum_{i=1}^R c_i S_{<d}(\ww{t}_i)$$
for some $c_i\in K$ and positive indices $\ww{t}_i\in\Ical$ independent of $d$. In particular, every Thakur multiple zeta value at a non-positive index can be expressed as a $K$-linear combination of Thakur multiple zeta values at positive indices.
\end{lem}
\begin{proof}
    Let $\ww{s}=(\ww{s}_1,s_k,\ww{s}_2)\in\Ical^{\ext}$ be a non-positive index with $\ww{s}_1=(s_1,\ldots,s_{k-1})\in\Ical$, $s_k\in\ZZ_{\leq 0}$ and $\ww{s}_2=(s_{k+1},\ldots,s_m)\in\Ical^{\ext}$. Then there exists $M\geq 0$ such that $S_{d_k}(s_k)=0$ for all $d_k> M$ by \cite[Theorem 5.1.2]{Thakur2004}. Thus, for all $d$ large enough, we have
\begin{align*}
   S_{<d}(\ww{s})&=\sum_{d_k=0}^{M}S_{d_k}(s_k)\left(\sum_{d>d_1>\cdots >d_{k-1}>d_k}S_{d_1}(s_1)\cdots S_{d_{k-1}}(s_{k-1})\right)\left(\sum_{d_k>d_{k+1}>\cdots>d_m\geq 0}S_{d_{k+1}}(s_{k+1})\cdots S_{d_{m}}(s_{m}) \right) \\&=
   \sum_{d_k=0}^{M}S_{d_k}(s_k)S_{<d_{k}}(\ww{s_2})\left(\sum_{d>d_1>\cdots>d_{k-1}>d_k}S_{d_1}(s_1)\cdots S_{d_{k-1}}(s_{k-1})\right).
\end{align*}

Now, for all positive index $\ww{r}=(r_1,\ldots,r_n)\in\Ical$ and $d'\in\ZZ$, we define $$F_{<d}(\ww{r},d')=\sum_{d>d_1>\cdots d_{n}>d'}S_{d_1}(r_1)\cdots S_{d_{n}}(r_{n})\quad\text{and}\quad S_{\leq d'}(\ww{r})=\sum_{d'\geq d_1>\cdots>d_n\geq 0}S_{d_1}(r_1)\cdots S_{d_{n}}(r_{n}).$$ 
It suffices to check that for all $d'\in\ZZ_{\geq 0}$ and $d\in\NN$, we have $$F_{<d}(\ww{r},d')=\sum_{i=1}^{R} c_iS_{<d}(\ww{t}_i)$$
for some $c_i\in K$ and $\ww{t}_i\in \Ical$.
We proceed by induction on the depth of $\ww{r}$. When $\dep(\ww{r})=1$ and $\ww{r}=(r_1)$, note that
$$F_{<d}(\ww{r},d')=\sum_{d>d_1\geq 0}S_{d_1}(r_1)-\sum_{d'\geq d_1\geq 0}S_{d_1}(r_1)=S_{<d}(\ww{r})-S_{\leq d'}(\ww{r})S_{<d}(\varnothing).$$
Assume the result holds for all $\ww{r}\in\Ical$ with $\dep(\ww{r})\leq N-1$. For an index $\ww{r}=(r_1,\ldots,r_N)\in\Ical$ with $\dep(\ww{r})=N$, we partition the summation domain $d>d_1>\cdots>d_N\geq 0$ based on the position of the threshold $d'$. There exists a unique $i \in \{0, 1, \ldots, N\}$ such that $d_i > d' \geq d_{i+1}$ (setting $d_0=d$ and $d_{N+1}=-1$ by convention). This yields that
$$S_{<d}(\ww{r}) = \sum_{i=0}^N S_{\leq d'}(\ww{r}_{> i})F_{<d}(\ww{r}_{\leq i}, d'),$$
where $\ww{r}_{\leq i} = (r_1, \ldots, r_i)$ and $\ww{r}_{>i} = (r_{i+1}, \ldots, r_N)$ (with the convention that $F_{<d}(\varnothing, d')=S_{\leq d'}(\varnothing)=1$). Thus, we obtain
$$F_{<d}(\ww{r}, d') = S_{<d}(\ww{r}) - \sum_{i=0}^{N-1} S_{\leq d'}(\ww{r}_{> i})F_{<d}(\ww{r}_{\leq i}, d').$$
Therefore, the result follows from the induction hypothesis.
\end{proof}

We now consider the formal variable $X$ and the new bracket defined by
\begin{equation*}
    [a]_{X} := a + a^r X = a(1 + a^{r-1}X) \in \CC_\infty\mbb{X}.
\end{equation*}
Then one checks the following:
\begin{enumerate}
    \item For all $a\in A_{+}$, $[a]_X\in\CC_\infty\mbb{X}^\times$.
    \item For all \(a,b\in A\),
    $$[a+b]_X=(a+b)+(a+b)^rX=(a+a^rX)+(b+b^rX)=[a]_X+[b]_X.$$
    \item For all $a\in A$ and $\varepsilon\in\FF_r$,
    $$[\varepsilon a]_X=\varepsilon a+(\varepsilon a)^r X=\varepsilon a+\varepsilon a^r X=\varepsilon [a]_X.$$
\end{enumerate}
Furthermore, let $\H_d^X(\ww{s})$ and $\H_{<d}^X(\ww{s})$ be defined as \eqref{eq:H_d def} and \eqref{eq:H_<d def} with respect to this bracket.

\begin{lem}\label{lem:v-bracket properties}
    For $d\in\NN$ and non-empty positive index $\ww{s}=(s_1,\ldots,s_m)\in\Ical$, we have 
    \begin{align*}
         \H_{<d}^X(\ww{s})= \sum_{N=0}^{\infty}  \sum_{\substack{n_1, \ldots, n_m \ge 0 \\ n_1+\dots+n_m=N}} \left( \prod_{j=1}^m \binom{-s_j}{n_j} \right) S_{<d}(s_1-n_1(r-1),\ldots,s_m-n_m(r-1))X^N.
    \end{align*}
    Therefore, as $d\to\infty$, it converges to
    $$\zeta_X(\ww{s}) := \lim_{d \to \infty} \H_{<d}^X(\ww{s})=\sum_{N=0}^{\infty} \Dcal_N(x_{\ww{s}}) X^N.$$
\end{lem}
\begin{proof}
    Let $\ww{s} = (s_1, \ldots, s_m) \in \Ical$. For each $a_j \in A_+$, we expand $[a_j]_X^{-s_j}$ in $\CC_\infty\mbb{X}$:
        $$[a_j]_X^{-s_j} = a_j^{-s_j}(1 + a_j^{r-1}X)^{-s_j} = \sum_{n_j=0}^{\infty} \binom{-s_j}{n_j} a_j^{-s_j+n_j(r-1)} X^{n_j}.$$
        Thus, we obtain
        \begin{align*}
            \frac{1}{[a_1]_X^{s_1}\cdots [a_m]_X^{s_m}} &= \prod_{j=1}^m \left( \sum_{n_j=0}^{\infty} \binom{-s_j}{n_j} a_j^{-s_j+n_j(r-1)} X^{n_j} \right) \\
            &= \sum_{N=0}^{\infty} X^N \sum_{\substack{n_1, \ldots, n_m \ge 0 \\ n_1+\dots+n_m=N}} \left( \prod_{j=1}^m \binom{-s_j}{n_j} \right) \frac{1}{a_1^{s_1-n_1(r-1)}\cdots a_m^{s_m-n_m(r-1)}}.
        \end{align*}
        Substituting this expansion into the definition of the finite multiple harmonic series $\H_{<d}^X(\ww{s})$ evaluated with the $X$-bracket, we get
        \begin{align*}
            \H_{<d}^X(\ww{s}) &= \sum_{\substack{a_1,\ldots,a_m\in A_+ \\ d>\deg a_1>\cdots>\deg a_m\geq 0}} \frac{1}{[a_1]_X^{s_1}\cdots [a_m]_X^{s_m}} \\
            &= \sum_{N=0}^{\infty} X^N \sum_{\substack{n_1, \ldots, n_m \ge 0 \\ n_1+\dots+n_m=N}} \left( \prod_{j=1}^m \binom{-s_j}{n_j} \right) \sum_{\substack{a_1,\ldots,a_m\in A_+ \\ d>\deg a_1>\cdots>\deg a_m\geq 0}} \frac{1}{a_1^{s_1-n_1(r-1)}\cdots a_m^{s_m-n_m(r-1)}}\\
            &= \sum_{N=0}^{\infty}  \sum_{\substack{n_1, \ldots, n_m \ge 0 \\ n_1+\dots+n_m=N}} \left( \prod_{j=1}^m \binom{-s_j}{n_j} \right) S_{<d}(s_1-n_1(r-1),\ldots,s_m-n_m(r-1))X^N.
        \end{align*} 
        Thus, $\H_{<d}^X(\ww{s})$ converges coefficient-wise to
        $$\zeta_X(\ww{s}) = \sum_{N=0}^{\infty} \Dcal_N(x_{\ww{s}}) X^N$$
        in $\CC_\infty\mbb{X}$ as $d\to\infty$.
\end{proof}

Therefore, we can apply Corollary \ref{cor:Z hat shuffle} to the formal series $$\zeta_X(\ww{s}) := \lim_{d \to \infty} \H_{<d}^X(\ww{s})$$
and obtain the following theorem:

\begin{thm}\label{thm:D_N_derivation_rigorous}
    The sequence of operators $(\Dcal_N)_{N\geq 0}$ forms a Hasse-Schmidt derivation over the realization map $\widehat{\zeta_A}:\Rcal\to\Zcal_\infty$. That is, for any indices $\ww{r}, \ww{s} \in \Ical$ and any integer $N \ge 0$, we have
    $\Dcal_0(x_{\ww{s}})=\zeta_A(\ww{s})$ and
    $$\Dcal_N(x_{\ww{r}} \ast x_{\ww{s}}) = \sum_{k=0}^N \Dcal_k(x_{\ww{r}}) \Dcal_{N-k}(x_{\ww{s}}).$$
\end{thm}
\begin{proof}
    It is clear from the definitions that $\Dcal_0=\widehat{\zeta_A}$. It suffices to check that $(\Dcal_N)_{N\geq 0}$ satisfies the convolution property.
    By Lemma \ref{lem:v-bracket properties} and Corollary \ref{cor:Z hat shuffle}, the map $\widehat{\zeta_X} : \Rcal \to \CC_\infty\mbb{X}$ sending $x_{\ww{s}} \mapsto \zeta_X(\ww{s})$ is an $\FF_p$-algebra homomorphism. Therefore, we have 
    $$\widehat{\zeta_X}(x_{\ww{r}} \ast x_{\ww{s}}) = \zeta_X(\ww{r}) \zeta_X(\ww{s}).$$
    Comparing the expansions of both sides, we obtain
    $$\sum_{N=0}^{\infty} \Dcal_N(x_{\ww{r}} \ast x_{\ww{s}}) X^N = \left( \sum_{k=0}^{\infty} \Dcal_k(x_{\ww{r}}) X^k \right) \left( \sum_{\ell=0}^{\infty} \Dcal_\ell(x_{\ww{s}}) X^\ell \right)=\sum_{N=0}^\infty \left(\sum_{k=0}^N \Dcal_k(x_{\ww{r}}) \Dcal_{N-k}(x_{\ww{s}})\right)X^N,$$
    which completes the proof.
\end{proof}

\begin{rmk}\label{rmk:419}
    Equations \eqref{eq:derivation} and \eqref{eq:N2_final_derivation}
    are special cases of Theorem \ref{thm:D_N_derivation_rigorous}
    for $\dep(\ww{r})=\dep(\ww{s})=1$ and $N=1,2$.
    This theorem shows that, by applying the derivations $\Dcal_N$ to the $r$-shuffle relations for positive indices, one can obtain relations involving Thakur's multiple zeta values at both positive and non-positive indices. After reducing the terms at non-positive indices, these relations can be reduced to relations among Thakur's power sums. For example, let $r=3$. Applying $\Dcal_9$ to
    $x_1\ast x_1$ gives
    \begin{align*}
    &\zeta_A(1)\zeta_A(-17)
    +\zeta_A(-1)\zeta_A(-15)
    +\zeta_A(-3)\zeta_A(-13)
    +\zeta_A(-5)\zeta_A(-11)
    +\zeta_A(-7)\zeta_A(-9)\\
    &=
    \zeta_A(1,-17)
    +\zeta_A(-1,-15)
    +\zeta_A(-3,-13)
    +\zeta_A(-5,-11)
    +\zeta_A(-7,-9)\\
    &\qquad
    +\zeta_A(-9,-7)
    +\zeta_A(-11,-5)
    +\zeta_A(-13,-3)
    +\zeta_A(-15,-1)
    +\zeta_A(-17,1)
    +2\zeta_A(-16).
    \end{align*}
    We now describe explicitly how this identity reduces in terms of
    Thakur's power sums. By the reduction used in the proof of
    Lemma \ref{lem:linearcombination-non-positive-index}, the relevant
    depth-one terms are
    \[
    \zeta_A(-m)
    =
    \begin{cases}
        1, & m=1,3,9,\\
        1+S_1(-m), & m=5,7,11,13,15,
    \end{cases}
    \]
    together with
    \[
    \zeta_A(-17)
    =
    1+S_1(-17)+S_2(-17)
    \]
    and
    \[
    \zeta_A(-16)
    =
    1+S_1(-16)+S_2(-16).
    \]
    For the depth-two terms, we have
    \[
    \zeta_A(-1,-15)
    =
    \zeta_A(-3,-13)
    =
    \zeta_A(-9,-7)
    =
    0,
    \]
\[
\zeta_A(-5,-11)=S_1(-5),\qquad
\zeta_A(-7,-9)=S_1(-7),
\]
and
\[
\zeta_A(-11,-5)=S_1(-11),\qquad
\zeta_A(-13,-3)=S_1(-13),\qquad
\zeta_A(-15,-1)=S_1(-15).
\]
    Moreover,
    \[
    \zeta_A(-17,1)
    =
    S_1(-17)
    +
    S_2(-17)\left(1+S_1(1)\right),
    \]
    and
    \[
    \begin{aligned}
    \zeta_A(1,-17)
    ={}&
    \zeta_A(1)-1+
    S_1(-17)
    \left(\zeta_A(1)-1-S_1(1)\right)+
    S_2(-17)
    \left(
        \zeta_A(1)-1-S_1(1)-S_2(1)
    \right).
    \end{aligned}
    \]
    Substituting all these expressions into the above identity and
    simplifying, we obtain
    \[
    S_1(1)S_1(-17)
    +
    S_2(1)S_2(-17)
    +
    S_1(-11)S_1(-5)
    +
    S_1(-16)
    +
    S_2(-16)
    =
    0.
    \]
    Thus, in this case, the relation obtained by applying $\Dcal_9$
    reduces to the above relation among Thakur's power sums.

    In general, reducing such derivation relations to relations among multiple zeta values at positive indices can be rather involved, and the resulting relations among Thakur's power sums can also be quite complicated.
\end{rmk}

We end this paper with the following remark.
\begin{rmk}
    While the present work focuses on the Carlitz module to construct the $u$-bracket, finite multiple harmonic $u$-series, and $u$-multiple zeta values, a natural and intriguing generalization is to consider \textit{Drinfeld $A$-modules}. We remark, however, that handling the convergence issues in this broader setting is expected to be considerably more difficult.
\end{rmk}
\sloppy

\bibliographystyle{amsplain}

\end{document}